\documentclass[final,12pt]{alt2022} 

\usepackage{times}

\usepackage[utf8]{inputenc} 
\usepackage[T1]{fontenc}    
\usepackage{hyperref}       
\usepackage{url}            
\usepackage{booktabs}       
\usepackage{amsfonts}       
\usepackage{nicefrac}       
\usepackage{microtype}      
\usepackage{xcolor}         

\usepackage{microtype}
\usepackage{graphicx}
\usepackage{booktabs} 
\usepackage{hyperref}

\usepackage{caption}
\usepackage{upgreek}
\usepackage{mathtools}
\usepackage{algorithm,algorithmic}
\usepackage{amssymb, multirow, paralist, color}
\usepackage{graphicx}
\usepackage{epstopdf}
\newtheorem{thm}{Theorem}

\newtheorem{lem}{Lemma}
\newtheorem{cor}{Corollary}

\newtheorem{ass}{Assumption}
\usepackage{enumitem}
\usepackage{booktabs}
\usepackage{amsmath}
\usepackage{bbm}
\usepackage{algorithm,algorithmic}

\def \X {\mathcal{X}}

\def \R {\mathbb{R}}

\def \w {\mathbf{w}}

\def \x {\mathbf{x}}

\def \x {\mathbf{x}}
\def \u {\mathbf{u}}

\def \p {\mathbf{p}}

\def \y {\mathbf{y}}
\def \u {\mathbf{u}}

\def \g {\mathbf{g}}

\def \y {\mathbf{y}}

\def \x {\mathbf{x}}
\def \g {\mathbf{g}}

\def \u {\mathbf{u}}

\def \w {\mathbf{w}}

\def \R {\mathbb{R}}

\DeclareMathOperator{\Regret}{R}

\def \N {\mathcal{N}}

\def \q {\mathbf{q}}

\def \p {\mathbf{p}}
\def \q {\mathbf{q}}

\def \gh {\widehat{\g}}

\def \X {\mathcal{X}}

\def \btheta {\boldsymbol{\theta}}

\def \bpi {\boldsymbol{\pi}}
\title[On the Initialization for Convex-Concave Min-max Problems]{On the Initialization for Convex-Concave Min-max Problems}
\usepackage{times}

\altauthor{%
 \Name{Mingrui Liu} \Email{mingruil@gmu.edu}\\
 \addr Department of Computer Science, George Mason University, Fairfax, VA 22030, USA
 \AND
 \Name{Francesco Orabona} \Email{francesco@orabona.com}\\
 \addr Electrical \& Computer Engineering, Boston University, Boston, Massachusetts 02215, USA
}

\begin{document}
	
	\maketitle
	\vspace*{-0.2in}
	\begin{abstract}%
		Convex-concave min-max problems are ubiquitous in machine learning, and people usually utilize first-order methods (e.g., gradient descent ascent) to find the optimal solution. One feature which separates convex-concave min-max problems from convex minimization problems is that the best known convergence rates for min-max problems have an explicit dependence on the size of the domain, rather than on the distance between initial point and the optimal solution. This means that the convergence speed does not have any improvement even if the algorithm starts from the optimal solution, and hence, is oblivious to the initialization. Here, we show that strict-convexity-strict-concavity is sufficient to get the convergence rate to depend on the initialization. We also show how different algorithms can asymptotically achieve initialization-dependent convergence rates on this class of functions. Furthermore, we show that the so-called ``parameter-free'' algorithms allow to achieve improved initialization-dependent asymptotic rates without any learning rate to tune. In addition, we utilize this particular parameter-free algorithm as a subroutine to design a new algorithm, which achieves a novel non-asymptotic fast rate for strictly-convex-strictly-concave min-max problems with a growth condition and H{\"o}lder continuous solution mapping. Experiments are conducted to verify our theoretical findings and demonstrate the effectiveness of the proposed algorithms. 
	\end{abstract}
	

	\section{Introduction}
	In this paper, we are interested in the following problem
	\begin{equation}
		\label{eq:problem}
		\min_{\x\in\mathcal{X}}\max_{\y\in\mathcal{Y}}\ F(\x,\y),  
	\end{equation}
	where $\mathcal{X}\subset\R^m$, $\mathcal{Y}\subset\R^n$ are convex and compact sets.
	This problem has broad applications in machine learning, e.g., stochastic AUC maximization~\citep{ying2016stochastic}, generative adversarial nets~\citep{goodfellow2014generative}, robust optimization~\citep{ben2009robust}, and adversarial training~\citep{madry2017towards}. Here, we will assume that $F(\x,\y)$ is convex in $\x$ and concave in $\y$.
	
	The canonical method for solving the convex-concave game is the primal-dual gradient method (a.k.a., gradient descent ascent)~\citep{nemirovskij1983problem,nemirovski2009robust}. For example, in the Euclidean setup, primal-dual gradient method is performing gradient descent on the primal variable $\x$ and gradient ascent on the dual variable $\y$ simultaneously. It is proved ~\citep{nemirovskij1983problem,nemirovski2009robust} that the averaged solution of primal-dual gradient method has good convergence guarantees in terms of the duality gap, which is
	\begin{equation}
		\label{eq:minmax}
		\max_{\y\in\mathcal{Y}} F(\bar{\x}_T, \y) - \min_{\x\in\mathcal{X}} F(\x,\bar{\y}_T)\leq \frac{D^2}{\eta T}+\eta G^2,
	\end{equation}
	where $(\bar{\x}_T, \bar{\y}_T)=(\frac{1}{T}\sum_{t=1}^{T}\x_t, \frac{1}{T}\sum_{t=1}^{T}\y_t)$ is the averaged solution, $D$ is the diameter of the domain, and $G$ is an upper bound on the norm of the  gradient.
	To get the tightest bound for the RHS of~(\ref{eq:minmax}), we can choose $\eta=\frac{D}{G\sqrt{T}}$ and end up with $\frac{DG}{\sqrt{T}}$ convergence rate for the duality gap.

	Strangely enough, the optimal learning rate scheme in the primal-dual gradient method, as well as its corresponding convergence rate, have an explicit dependence on the size of the domain and are completely independent of the initialization. In other words, regardless of how close the initialization is to the optimal solution, this algorithm uses the same learning rate and ends up with the same complexity guarantees. This is counter-intuitive: we would expect to be able to obtain a faster convergence if the initialization is closer to the optimal solution. In other words, we aim to obtain an initialization-dependent convergence rate, i.e., $\widetilde{\mathcal{O}}\left(\frac{f(\text{dist}(\x_0,\x_*)+f(\text{dist}(\y_0,\y_*))}{\sqrt{T}}\right)$ rate,\footnote{The $\widetilde{\mathcal{O}}$ notation hides poly-logarithmic terms. In the case that $\x_0=\x_*$, $\y_0=\y_*$, we allow the big O notation to hide lower order terms such as $\mathcal{O}(1/T)$.} where $\text{dist}$ is some metric depending on the geometry of the problem, and $f(\cdot):\R^+\mapsto\R^+$ is some non-decreasing function depending on the algorithm with $f(0)=0$. Yet, the hardness results in the optimization literature show that $\Omega(\frac{DG}{\sqrt{T}})$ complexity lower bound is unimprovable in general~\citep{nemirovskij1983problem,juditsky2011solving}. This naturally motivates the following questions:
	
	\textit{What is the suitable class of the functions such that one can use first-order algorithms to get the initialization-dependent rates? In addition, how to utilize this fact to design better algorithms with faster rates for min-max problems?}
	

	
	
	The main goal of this paper is to answer these questions. We show that strict-convexity-strict-concavity is the key to derive initialization-dependent rates as well as fast rates of first-order algorithms for solving min-max problems.

	More in details, our contributions are summarized as follows.
	\begin{itemize}
		\item We identify a condition (i.e., the problem instance is a strictly-convex-strictly-concave min-max problem), and we show that this condition is sufficient to obtain asymptotic initialization-dependent rates for any first-order algorithm with a generic convergence guarantee.
		\item We show that invoking a parameter-free algorithm~\citep{OrabonaP16,CutkoskyO18} on both primal and dual variables can achieve asymptotic initialization-dependent rates without any learning rate to tune. Taking into account the known upper bounds of gradient descent ascent without additional assumptions on the curvature, it remains unclear how to obtain an initialization-dependent rate unless we have the knowledge  of the distance between initialization and the optimal solution.
		
		\item When the function admits a growth condition and H{\"o}lder continuous solution mapping, we design a new algorithm by utilizing the parameter-free algorithm as a subroutine and periodically restarting the subroutine. Thanks to the initialization-dependent rate, we prove that our algorithm enjoys novel fast non-asymptotic rates. To the best of our knowledge, this is the first work leveraging the function growth condition and H{\"o}lder continuous solution mapping to obtain these improved non-asymptotic rates for min-max problems.
		\item We verify our theoretical results by conducting both synthetic experiments and distributionally robust optimization on benchmark datasets. We empirically show that our algorithms exhibit good performance in practice.
	\end{itemize}
	
	\section{Related Work}
	\label{sec:rel}
	
	\noindent\textbf{Convex-Concave Min-max Optimization} Convex-concave Min-max Optimization is widely studied in optimization literature, and it is closely related to the variational inequality. The work of~\cite{korpelevich1976extragradient} proposed the extragradient method for solving variational inequalities, and this was later extended into non-euclidean space (e.g., mirror-prox~\citep{nemirovski2004prox}, dual extrapolation~\citep{nesterov2007dual}). The stochastic version of mirror-prox was proposed by~\citet{juditsky2011solving}. \citet{HsiehIMM20} proposed variable stepsize scaling for extragradient method to improve the algorithm's performance.  In nonsmooth case,~\citet{nemirovski2009robust} analyzed the primal-dual gradient method in non-euclidean space. \citet{nedic2009subgradient}  considered subgradient methods for solving min-max problems and provided per-iteration convergence rate estimates on the solutions. \citet{monteiro2010complexity} designed hybrid proximal extragradient methods with a different performance measure. \citet{bach2019universal} provided a universal algorithm for solving variational inequalities, which adapts to noise and smoothness.
	
	There are several papers considering specific cases in convex-concave min-max optimization, including functions with a bilinear term~\citep{nesterov2005smooth,chen2014optimal,chen2017accelerated,he2016accelerated,liu2018fast,daskalakis2018training,liang2019interaction,gidel2019negative,mokhtari2020unified,azizian2020accelerating,bailey2020finite},  smooth or strongly-convex-(strongly)-concave~\citep{nesterov2006solving,zhao2019optimal,lin2020near,yan2020optimal}, the last-iterate convergence~\citep{abernethy2019last,daskalakis2018last,golowich2020last}, adapts to unknown smoothness parameter~\citep{diakonikolas2020halpern} . There are also some papers about establishing lower bounds in various cases~\citep{zhang2019lower,ibrahim2020linear,ouyang2021lower}.
	
	However, none of these works provide an upper bound for duality gap which explicitly depends on the distance between the initialization and the optimal solution without assuming strong convexity/concavity~\citep{chambolle2011first}.
	
	
	
	\noindent\textbf{Parameter-Free Online Convex Optimization}
	In Online Convex Optimization (OCO)~\citep{Gordon99,Zinkevich03}, the aim of the learner is to minimize the regret w.r.t. any fixed predictor.  Most of the OCO algorithm require some knowledge of the competitor, for example, its norm, in order to achieve the smallest regret~\citep[see, e.g.,][]{Orabona19}. Hence, it becomes impossible to compete \emph{uniformly} with all competitors, unless the algorithm has some knowledge of the future. Morally speaking, the OCO setting is a strict generalization of the setting of stochastic optimization of convex functions and competing with any fixed predictor corresponds exactly to design convex optimization algorithms that have optimal dependency on the distance between the initial point and the optimal solution. Again, without some knowledge of the norm of the optimal solution, classic optimization algorithms fails to get the right dependency.
	
	So-called \emph{parameter-free} algorithms avoid setting step sizes completely and get the optimal dependency uniformly on all competitors~\citep{StreeterM12}. Note that ``parameter-free'' is a technical word that denotes only this ability of the algorithm and not a general lack of knowledge of \emph{any} characteristic of the problem, see \citet[Section 9.7]{Orabona19}.
	Most of these algorithms are based on Follow-the-Regularized-Leader\footnote{Dual Averaging~\citep{Nesterov09} is a specialization of FTRL to linear functions.} (FTRL)~\citep{Shalev-Shwartz07,AbernethyHR08} with a time-varying linearithmic (non-strongly convex) regularizer~\citep[e.g.,][]{StreeterM12,Orabona14,Cutkosky16,CutkoskyB17,Kotlowski19,KempkaKW19}. These algorithms can also be viewed as betting schemes through the duality between regret and reward~\citep{McMahanO14,OrabonaP16,CutkoskyO18,CutkoskyS19}.
	
	As far as we know, we present the first application of parameter-free algorithms to convex-concave min-max problems.
	\vspace*{-0.1in}
	\section{Problem Setup}
	
	\noindent\textbf{Notation}
	Denote by $\|\cdot\|$ the Euclidean norm. Define $\langle\cdot,\cdot\rangle$ as the inner product in Euclidean space.  Define $\Pi_{\mathcal{S}}$ as the orthogonal projection operator onto the set $\mathcal{S}$. We denote vectors by bold letters, e.g., $\x,\g$. We say that $f(x)=o(g(x))$ iff $\lim_{x \rightarrow \infty} \frac{f(x)}{g(x)}=0$.
	
	\noindent\textbf{Setting and Assumptions}
	In~(\ref{eq:problem}), we assume that $F$ is convex in the first argument and concave in the second one. Moreover, we assume to have access to a first-order black-box optimization oracle $\g=(\g^{\x},\g^{\y})$ at  any point $(\x,\y)$,
	where $\g^{\x}\in \partial_{\x} F(\x,\y)$, the subgradient of $F$ w.r.t. its first argument, and $\g^{\y} \in -\partial_{\y} (-F(\x,\y))$, the negative subgradient of $-F$ w.r.t. its second argument. 
	
	We will also assume that the subgradients $\g^{\x}$ and $\g^{\y}$ have bounded support. In particular, we assume w.l.o.g. that $\|\g^{\x}\|\leq 1$ and $\|\g^{\y}\|\leq 1$ for all $\x \in \mathcal{X}$ and $\y \in \mathcal{Y}$. Note that it is reasonable because we assume bounded domains.

	\paragraph{An Sufficient  Condition for Initialization-Dependent Rates: the Class of Strictly-Convex-Strictly-Concave Min-Max Problems}

	In this paper, we aim to show that we can achieve initialization-dependent convergence rates for the class of strictly-convex-strictly-concave min-max problems (which covers the class of strongly-convex-strongly-concave min-max problems as a subclass), which means that $F(\x,\y)$ is strictly convex in $\x$ and strictly concave in $\y$.  The definition of strictly convex function is the following.
	\begin{definition}
		A function $h$ is strictly convex if $h(\lambda\x+(1-\lambda)\y)<\lambda h(\x)+(1-\lambda)h(\y)$ for any $0<\lambda<1$ and any $\x\neq \y$.
	\end{definition}
	
	Given the above arguments, in the following we will assume that the following assumption holds:
	\begin{ass}
		\label{ass:2}
		$F(\x,\y)$ is strictly convex in $\x$ and strictly concave in $\y$ and we denote by $(\x_*,\y_*)$ the optimal solution of the min-max problem~\eqref{eq:problem}. $F(\x,\y)$ is continuous in $\y$ (resp. $\x$) given $\x$ (resp. $\y$).
	\end{ass}
	
	
	\section{Algorithms with Initialization-Dependent Convergence Rates}
	\label{sec:initialization-rate}
	In this section, we first present a general convergence theory for first-order algorithms to obtain initialization-dependent convergence rates, under the assumption that the function is strictly-convex-strictly-concave (in Section~\ref{sec:general}). We also show conditions under which we get initialization-dependent convergence rates for gradient descent ascent~(in Section~\ref{sec:GDAalgo}).
	Then, we show a parameter-free algorithm that always enjoys an initialization-dependent convergence rate~(in Section~\ref{sec:pmalgo}). Finally we discuss possible alternative approches to obtain initialization-dependent rate~(in Section~\ref{sec:discussion}).
	
	
	\subsection{General Convergence Theory for First-order Algorithms}
	\label{sec:general}
	
	In this section, we present our main theorem (Theorem~\ref{main:metathm}), which characterizes a general convergence theory for first-order algorithms to achieve initialization-dependent asymptotic rates.
	
	\begin{thm}
		\label{main:metathm}
		Let $D_\mathcal{X}$ be the diameter of $\mathcal{X}$ and $D_\mathcal{Y}$ is the diameter of $\mathcal{Y}$, and $D=\max(D_{\mathcal{X}},D_{\mathcal{Y}})$.
		Suppose Assumption~\ref{ass:2} holds, and there is an algorithm $\mathcal{A}$ which returns a solution $(\widetilde{\x}_T,\widetilde{\y}_T)$ after $T$ iterations with following guarantee:
		\begin{equation}
			\label{eq:boundgeneral1}
			\max_{\y\in\mathcal{Y}} F(\widetilde{\x}_T, \y) - \min_{\x\in\mathcal{X}} F(\x,\widetilde{\y}_T)
			\leq A(\|\x_0-\x'_T\|)B(T)+A(\|\y_0-\y'_T\|)B(T)+C(T),
		\end{equation}
		where $\x'_T=\arg\min_{\x}F(\x,\widetilde{\y}_T)$, $\y'_T=\arg\max_{\y}F(\widetilde{\x}_T,\y)$, $A:\R_+\to \R$ is non-decreasing, $A(0)=0$, $B: \N \to\R_{++}$ and $C:\N \to\R_{++}$ are non-increasing, $\lim\limits_{T\rightarrow\infty} B(T)=0$, $\lim\limits_{T\rightarrow\infty} C(T)=0$, $A(x+y)\leq c\left(A(x)+A(y)\right)$ with $c$ being a positive constant independent of $D$. Then, we have
		\begin{equation}
			\label{eq:boundgeneral} 
			\max_{\y\in\mathcal{Y}} F(\widetilde{\x}_T, \y) - \min_{\x\in\mathcal{X}} F(\x,\widetilde{\y}_T)
			\leq c B(T) (A(\|\x_0-\x_*\|) +  A(\|\y_0-\y_*\|)) + R(T) + C(T),
		\end{equation}
		where the residual term $R: \N \rightarrow \R_+$ has the following properties:
		\[
		R(t) = o(B(t)) \quad \text{and} \quad R(T) \leq 2 c\cdot A(D)B(T)~.
		\]
	\end{thm}
	Theorem~\ref{main:metathm} indicates that if we have any algorithm with guarantees of form~(\ref{eq:boundgeneral1}), then under Assumption~\ref{ass:2}, it automatically enjoys an initialization-dependent asymptotic convergence rate as illustrated by~(\ref{eq:boundgeneral}).
	
	The Assumption~\ref{ass:2} is crucial to get the desired bound due to the following reasons. First, it can guarantee that both $\x'_T$ and $\y'_T$ are unique. Second, we can utilize this assumption to show that $\x'_T$ ($\y'_T$) gets close to $\x_*$ ($\y_*$) when the algorithm runs a sufficiently large number of iterations, which ends up with the $o(B(T))$ term in the bound~(\ref{eq:boundgeneral}). 
	
	\begin{proof}[Proof of Theorem~\ref{main:metathm}]
		According to the fact that $A(x)$ is non-decreasing and \eqref{eq:boundgeneral1}, and noting that $\|\x_0-\x'_T\|\leq \|\x_0-\x_*\|+\|\x_*-\x'_T\|$, $\|\y_0-\y'_T\|\leq \|\y_0-\y_*\|+\|\y_*-\y'_T\|$, we have
		\begin{equation}
			\begin{aligned}
				&\max_{\y\in\mathcal{Y}} F(\widetilde{\x}_T, \y) - \min_{\x\in\mathcal{X}} F(\x,\widetilde{\y}_T)\\
				&\quad \leq A(\|\x_0-\x_*\|+\|\x_*-\x'_T\|)B(T)+A( \|\y_0-\y_*\|+\|\y_*-\y'_T\|)B(T)+C(T)\\
				&\quad \leq c B(T)\left[A(\|\x_0-\x_*\|)+A(\|\y_0-\y_*\|) + A(\|\x_*-\x'_T\|)+A(\|\y_*-\y'_T\|)\right]+C(T),
			\end{aligned}
		\end{equation}
		where the last inequality holds due to the fact that $A(x+y)\leq c(A(x)+A(y))$. It remains to show that $A(\|\x_*-\x'_T\|)+A(\|\y_*-\y'_T\|)\rightarrow 0$ when $T\rightarrow\infty$. Since $A(0)=0$, it suffices to show that  $\|\x_*-\x'_T\|\rightarrow 0$ and $\|\y_*-\y'_T\|\rightarrow 0$.
		
		Note that 
		\begin{align}
			F(\x_*,\y_*)-F(\x_*,\widetilde{\y}_T)
			&\leq F(\x_*,\y_*)-\min_{\x\in\mathcal{X}}F(\x,\widetilde{\y}_T)
			\leq \max_{\y\in\mathcal{Y}}F(\widetilde{\x}_T,\y)-\min_{\x\in\mathcal{X}}F(\x,\widetilde{\y}_T)+C(T) \nonumber \\
			&\leq A(\|\x_0-\x'_T\|)B(T)+A(\|\y_0-\y'_T\|)B(T) + C(T) \nonumber \\
			&\leq A(D_\mathcal{X})B(T)+A(D_\mathcal{Y})B(T)+C(T) 
			\leq 2A(D)B(T)+C(T)~. \label{eq:2:lemma2}
		\end{align}
		
		(i). We now claim that $\widetilde{\y}_T\rightarrow\y_*$ when $T\rightarrow \infty$. Let's see why.
		
		Note that $F(\x,\y)$ is strictly concave in terms of $\y$. Define $\widetilde{F}(\y)=F(\x_*,\y)$. Taking the limit for $T\rightarrow \infty$ in~(\ref{eq:2:lemma2}), for the sequence $\{\widetilde{\y}_t\}_{t=1}^{\infty}$, we have $\widetilde{F}(\widetilde{\y}_t)\rightarrow \widetilde{F}(\y_*)$. Given that the domain is bounded, we can extract a convergent subsequence $\{\widetilde{\y}_{t_j}\}\subset\{\widetilde{\y}_t\}$ and we assume that $\widetilde{\y}_{t_j}\rightarrow \widetilde{\y}$. By the continuity of $\widetilde{F}(\y)$ in terms of $\y$, we know that $\widetilde{F}(\bar{\y}_{t_j})\rightarrow \widetilde{F}(\widetilde{\y})$. Now, $\widetilde{F}(\widetilde{\y}_{t_j})$ is also a subsequence of the convergent sequence $\widetilde{F}(\widetilde{\y}_t)$, then $\widetilde{F}(\widetilde{\y}_{t_j})\rightarrow \widetilde{F}(\y_*)$. Since $\y_*$ is uniquely defined, this implies that $\widetilde{\y}=\y_*$. This means that any convergent subsequence of $\{\widetilde{\y}_t\}_{t=1}^{\infty}$ converges to $\y_*$, so $\widetilde{\y}_T\rightarrow\y_*$ when $T\rightarrow\infty$.
		
		(ii). Our next claim is that the mapping $\arg\min_{\x\in\mathcal{X}} F(\x,\y)$ is a continuous function in $\y$. 
		
		First, define $H(\y)=\arg\min_{\x\in\mathcal{X}} \ F(\x,\y)$. By the compactness of $\mathcal{Y}$, we have a sequence $\y_k\rightarrow\y_*$. Define $\x_k=H(\y_k)$ and $\x_*=H(\y_*)$. By the compactness of $\mathcal{X}$, there exists a convergent subsequence $\x_{k_i}$ of $\x_k$, and we denote its limit by $\widetilde{\x}$. From the above, we have $F(\x_{k_i},\y_{k_i})\leq F(\x,\y_{k_i})$ for all $\x\in\mathcal{X}$, that implies $F(\widetilde{\x},\y_*)\leq F(\x,\y_*)$ for any $\x\in\mathcal{X}$. In particular, this implies $F(\widetilde{\x},\y_*)\leq F(\x_*,\y_*)$. By the uniqueness of the minimizer, we must have $\widetilde{\x}=\x_*$. Given that any convergent subsequence converges to $\x_*$, this means that $\x_k\rightarrow\x_*$ and hence $\arg\min_{\x\in\mathcal{X}} F(\x,\y)$ is a continuous function in terms of $\y$.
		
		Combining the above two claims (i) and (ii), and by the definition of $\x_*$ and $\x'_T$, we obtain that $\x'_T\rightarrow \x_*$ when $T\rightarrow\infty$. Hence $\|\x_*-\x'_T\|\rightarrow 0$. A parallel argument can guarantee that $\|\y_*-\y'_T\|\rightarrow 0$. This completes the proof.
	\end{proof}
	
	\subsection{Initialization-dependent Rate by Gradient Descent Ascent}
	\label{sec:GDAalgo}
	In this section, we show that the standard algorithm gradient descent ascent satisfies~(\ref{eq:boundgeneral1}). The update rule of gradient descent ascent is
	\[
	\begin{cases}
		\x_{t+1}=\Pi_{\mathcal{X}}\left[\x_t-\eta\g^{\x}(\x_t,\y_t)\right]\\
		\y_{t+1}=\Pi_{\mathcal{Y}}\left[\y_t+\eta\g^{\y}(\x_t,\y_t)\right],
	\end{cases}
	\]
	where $\g^{\x}(\x_t,\y_t)\in\partial_{\x}F(\x_t,\y_t)$, $\g^{\y}(\x_t,\y_t)\in-\partial_{\x}\left(-F(\x_t,\y_t)\right)$, $\Pi$ is the projection operator, $\eta$ is the learning rate, and $\partial_{\x}$, $\partial_{\y}$ denote the subdifferential in terms of $\x$ and $\y$ respectively. If we use $\eta=\alpha/\sqrt{T}$ and run the gradient descent ascent for $T$ iterations, the standard theory of gradient descent ascent~\citep{nemirovski2009robust} shows that
	\[
	\max_{\y\in\mathcal{Y}} F(\bar{\x}_T, \y) - \min_{\x\in\mathcal{X}} F(\x,\bar{\y}_T)
	\leq A(\|\x_0-\x'_T\|)B(T)+A(\|\y_0-\y'_T\|)B(T)+C(T),
	\]
	where $A(x)=\frac{x^2}{2\alpha}$, $B(T)=\frac{1}{\sqrt{T}}$, $C=\frac{\alpha}{\sqrt{T}}$, $\x'_T=\arg\min_{\x\in\mathcal{X}}F(\x,\y_T)$, and $\y'_T=\arg\min_{\y\in\mathcal{Y}}F(\bar{\x}_T,\y)$. We can verify that it satisfies the premises of Theorem~\ref{main:metathm} (note that $A(x+y)\leq 2A(x)+2A(y)$). Hence, invoking Theorem~\ref{main:metathm} yields
	\begin{equation*}
		\begin{aligned}
			\max_{\y\in\mathcal{Y}} F(\bar{\x}_T, \y) - \min_{\x\in\mathcal{X}} F(\x,\bar{\y}_T)
			&\leq 2 A(\|\x_0-\x_*\|)B(T)+ 2 A(\|\y_0-\y_*\|)B(T)+R(T)+\frac{\alpha}{\sqrt{T}}\\
			& =\frac{\|\x_0-\x_*\|^2+\|\y_0-\y_*\|^2}{\alpha \sqrt{T}}+\frac{\alpha}{\sqrt{T}}+R(T),
		\end{aligned}
	\end{equation*}
	where $R(t) = o(\frac{1}{\sqrt{t}})$ and $R(T) \leq \frac{2 D^2}{\alpha\sqrt{T}}$.
	
	From the above, we see that the hyperparameter $\alpha$ in the learning rate plays a major role in the convergence rate. There are 3 possible choices for $\alpha$.
	\begin{itemize}
		\item The optimal setting of $\alpha$ would be $\sqrt{\|\x_0-\x_*\|^2+\|\y_0-\y_*\|^2}$, that could give an initialization-dependent convergence rate of $\mathcal{O}(\frac{\sqrt{\|\x_0-\x_*\|^2+\|\y_0-\y_*\|^2}}{\sqrt{T}})$. However, it is easy to realize that such setting is impossible in the black-box optimization setting: The algorithm has no way to estimate the initial condition, not even in the convex optimization case. Hence, for any fixed choice of $\alpha$ there exists an optimization on which this choice is suboptimal. Indeed, there are little-known lower bounds in the OCO setting that hints to the fact that the rate $\mathcal{O}(\frac{\sqrt{\|\x_0-\x_*\|^2+\|\y_0-\y_*\|^2}}{\sqrt{T}})$ might be actually \emph{impossible}~\citep{StreeterM12} and \emph{an additional polylogarithmic price should be paid to estimate the initial condition}, for example trying a logarithmic number of possible settings of $\alpha$ and selecting the best one.
		\item The worst-case choice and the standard choice for this algorithm~\citep{nemirovski2009robust} is to set $\alpha=D$ and obtain a finite-time rate of $\frac{D}{\sqrt{T}}$, without any asymptotic acceleration.
		\item The last choice we have is to set $\alpha$ to any arbitrary number and have an asymptotic rate of $\mathcal{O}(\frac{\frac{\|\x_0-\x_*\|^2+\|\y_0-\y_*\|^2}{\alpha}+\alpha}{\sqrt{T}})$. This is the only viable choice that gives an initialization-dependent rate on all problems such that $\alpha< \sqrt{\|\x_0-\x_*\|^2+\|\y_0-\y_*\|^2}$.
	\end{itemize}
	
	In the next section, we will show how parameter-free algorithms can get initialization-dependent asymptotic rate without any hyperparameter tuning.

	\subsection{Improved Initialization-Dependent Convergence by Parameter-free Algorithms}
	\label{sec:pmalgo}
	
	In this section, we establish an improved initialization-dependent convergence by parameter-free algorithms for solving min-max problems.
	The core idea of this approach is to decouple the primal problem and dual problem in~(\ref{eq:problem}), and by then utilize no-regret algorithms to perform the optimization. This is not a new idea by any means \citep[see, for example,][]{AbernethyLLW18}. However, we need to be particularly careful of the dependency on the initial point. For this reason, we first state a Theorem to use no-regret algorithms in min-max problems that is a simple generalization of Theorem 9 in~\citep{AbernethyLLW18}. In particular, here we care about considering the regret of the OCO algorithms w.r.t. specific points. This will be critical in obtaining initialization-dependent rates.
	
	Suppose to use an OCO algorithm fed with losses $\ell_t(\x)= F(\x,\y_t)$ that produces the iterates $\x_t$ and another OCO algorithm fed with losses $h_t(\y)=-F(\x_t,\y)$ that produces the iterates $\y_t$. Then, we can state the following Theorem (for completeness the proof is in Appendix~\ref{appendix:mainsection}).
	\begin{thm}
		\label{prop:oco-to-minmax}
		Let $\bar{\x}_T=\frac{1}{T}\sum_{t=1}^{T} \x_t$, $\bar{\y}_T=\frac{1}{T}\sum_{t=1}^{T} \y_t$. Then, we have
		\[
		\max_{\y\in\mathcal{Y}} F(\bar{\x}_T, \y) - \min_{\x\in\mathcal{X}} F(\x,\bar{\y}_T)
		\leq \frac{\Regret_T(\x'_T)+\Regret_T(\y'_T)}{T},
		\]
		where $\Regret_T(\y)=\sum_{t=1}^{T}h_t(\y_t)-\sum_{t=1}^{T}h_t(\y)$,  $\Regret_T(\x)=\sum_{t=1}^{T}\ell_t(\x_t)-\sum_{t=1}^{T}\ell_t(\x)$, $\x_T'\in\arg\min_{\x\in\mathcal{X}}F(\x,\bar{\y}_T)$, and $\y_T'\in\arg\max_{\y\in\mathcal{Y}}F(\bar{\x}_T,\y)$.
	\end{thm}
	In words, the above theorem says that we can use two OCO algorithms to minimize the problem in~\eqref{eq:minmax}. In particular, the convergence rate depends on the sum of the regret of the two OCO algorithms evaluated in specific points, divided by the sum of the weights.
	
	\begin{algorithm}[t]
		\caption{Constrained Parameter-free OCO}
		\label{alg:constrained-cb}
		\begin{algorithmic}[1]
			\REQUIRE Convex and compact feasible set $\mathcal{X}$, initial point $\x_0\in\mathcal{X}$, number of iterations $T$
			\STATE $\widetilde{\x}_1=\x_0$
			\STATE $\btheta_t=\boldsymbol{0}$
			\FOR{$t=1,\dots, T$}
			\STATE $\x_t=\Pi_{\mathcal{X}}\left(\widetilde{\x}_t\right)$
			\STATE Receive subgradients $\widehat{\g}_t \in \partial \ell_t(\x_t)$ such that $\|\widehat{\g}_t\|\leq 1$
			\STATE $\g_t=\frac{1}{2}\left(\widehat{\g}_t+\|\widehat{\g}_t\|\cdot \frac{\widetilde{\x}_t-\x_t}{\|\widetilde{\x}_t-\x_t\|}\right)$ (Define $\mathbf{0}/0=\mathbf{0}$)
			\STATE $\widetilde{\x}_{t+1}=\x_0 - \frac{\sum_{i=1}^{t-1} \g_i}{t+1}\left(1 - \sum_{i=1}^{t-1} \langle\widetilde{\x}_{i}, \g_i\rangle\right)$
			\ENDFOR
		\end{algorithmic}
	\end{algorithm}
	
	Inspired by~\cite{OrabonaP16,CutkoskyO18}, we present a parameter-free algorithm for constrained optimization, in Algorithm~\ref{alg:constrained-cb}. For it, we can prove the following regret guarantee in Theorem~\ref{thm:constrained_new_algo} (the proof is in Appendix~\ref{appendix:mainsection}).
	\begin{thm}
		\label{thm:constrained_new_algo}
		Assume $\|\gh_t\|\leq1$ for all $t=1, \dots, T$. Then, for all $\u \in \mathcal{X}$, we have
		\[
		\sum_{t=1}^T (\ell_t(\x_t) - \ell_t(\u))
		\leq 2+2\|\x_0- \u\|\sqrt{ T \ln\left(24 T^2\|\x_0-\u\|^2+1\right)}~.
		\]
	\end{thm}
	As in other parameter-free algorithms, the regret depends on the optimal quantity $\widetilde{\mathcal{O}}(\|\x_0 - \u\|)$, rather than on the worse $\mathcal{O}(\|\x_0-\u\|^2)$ or $D$ of online gradient descent. Finally, the constrained set is dealt with the black-box reductions from unconstrained OCO to constrained OCO~\citep{CutkoskyO18}, that prescribes the use of the projections and surrogate gradients.
	
	Now, using Theorem~\ref{prop:oco-to-minmax}, we can use two instantiations of Algorithm~\ref{alg:constrained-cb} on the primal $\x$ and dual variable $\y$ to solve min-max problems. 
	Putting all together, we obtain Algorithm~\ref{alg:cb-minmax}.
	
	\begin{algorithm}[t]
		\caption{$\text{CB-Min-Max}(\x_0,\y_0,T)$}
		\label{alg:cb-minmax}
		\begin{algorithmic}[1]
			\REQUIRE $\x_0\in\mathcal{X},\y_0\in\mathcal{Y}$
			\STATE $\widetilde{\x}_1=\x_0,\widetilde{\y}_1=\y_0$
			\FOR{$t=1,\dots, T$}
			\STATE $\x_t=\Pi_{\mathcal{X}}\left(\widetilde{\x}_t\right)$, $\y_t=\Pi_{\mathcal{Y}}\left(\widetilde{\y}_t\right)$
			\STATE Receive subgradients $\widehat{\g}_t=(\widehat{\g}_t^{\x},\widehat{\g}_t^{\y})$
			\STATE $\g_t^{\x}=\frac{1}{2}\left(\widehat{\g}_t^{\x}+\|\widehat{\g}_t^{\x}\|\cdot \frac{\widetilde{\x}_t-\x_t}{\|\widetilde{\x}_t-\x_t\|}\right)$ (Define $\mathbf{0}/0=\mathbf{0}$)
			\STATE $\g_t^{\y}=\frac{1}{2}\left(\widehat{\g}_t^{\y}+\|\widehat{\g}_t^{\y}\|\cdot \frac{\widetilde{\y}_t-\y_t}{\|\widetilde{\y}_t-\y_t\|}\right)$ (Define $\mathbf{0}/0=\mathbf{0}$)
			\STATE $\widetilde{\x}_{t+1}=\x_0 - \frac{\sum_{i=1}^{t-1} \g^{\x}_i}{t+1}\left(1 - \sum_{i=1}^{t-1} \langle\widetilde{\x}_{i}, \g^{\x}_i\rangle\right)$
			\STATE $\widetilde{\y}_{t+1}=\y_0 + \frac{\sum_{i=1}^{t-1} \g^{\y}_i}{t+1}\left(1 + \sum_{i=1}^{t-1} \langle\widetilde{\y}_{i}, \g^{\y}_i\rangle\right)$
			\ENDFOR
			\STATE \textbf{Return} $\bar{\x}_T=\frac{1}{T}\sum_{t=1}^{T} \x_t$, $\bar{\y}_T=\frac{1}{T}\sum_{t=1}^{T} \y_t$
		\end{algorithmic}
	\end{algorithm}
	
	Using Theorem~\ref{thm:constrained_new_algo} and Theorem~\ref{prop:oco-to-minmax}, we are able to show the following Corollary (proof is in Appendix~\ref{appendix:mainsection}).
	\begin{cor}
		\label{cor:1}
		Algorithm~\ref{alg:cb-minmax} guarantees that 
		\begin{equation}
			\label{eq:cb1}
			\begin{aligned}
				&\max_{\y\in\mathcal{Y}} F(\bar{\x}_T, \y) - \min_{\x\in\mathcal{X}} F(\x,\bar{\y}_T)\\
				&\quad \leq \frac{4}{T}+ \frac{2\|\x_0-\x'_T\|\sqrt{\ln\left(24 T^2\|\x_0-\x'_T\|^2+1\right)}+2\|\y_0-\y'_T\|\sqrt{\ln\left(24 T^2\|\y_0-\y'_T\|^2+1\right)}}{\sqrt{T}},
			\end{aligned}
		\end{equation}
		where $\x_T'\in\arg\min_{\x\in\mathcal{X}}F(\x,\bar{\y}_T)$, and $\y_T'\in\arg\max_{\y\in\mathcal{Y}}F(\bar{\x}_T,\y)$.
	\end{cor}

	Now, from Corollary~\ref{cor:1} and Theorem~\ref{main:metathm}, we show asymptotic initialization-dependent convergence rates for Algorithm~\ref{alg:cb-minmax}. The proof of Theorem~\ref{thm:main} is included in Appendix~\ref{appendix:mainsection}.
	\begin{thm}
		\label{thm:main}
		Suppose Assumption~\ref{ass:2} holds. Then, Algorithm~\ref{alg:cb-minmax} for any $T$ iterations satisfies
		\begin{equation*}
			\begin{aligned}
				\max_{\y\in\mathcal{Y}}F(\bar{\x}_T,\y)-\min_{\x\in\mathcal{X}}F(\x,\bar{\y}_T)
				\leq\frac{2\left(\|\x_0-\x_*\|+\|\y_0-\y_*\|\right)\sqrt{\ln(24 T^2 D^2+1)}}{\sqrt{T}}+o\left(\frac{1}{\sqrt{T}}\right)~.
			\end{aligned}
		\end{equation*}
	\end{thm}
	
	\textbf{Remark}: Comparing Theorem~\ref{thm:main} and the argument in Section~\ref{sec:GDAalgo}, we can see that asymptotically CB-Min-Max has rate $\frac{\sqrt{\log T}}{\sqrt{T}}$, while gradient descent ascent has rate $\frac{1}{\sqrt{T}}$, so CB-Min-Max is worse than gradient descent ascent asymptotically. However, gradient descent ascent has initialization-dependent rate only if we have the knowledge of the initial condition, while CB-Min-Max always obtains initialization-dependent rate. We will show in Section~\ref{sec:extension1} that this particular feature of CB-Min-Max is important to derive non-asymptotic fast rates on another class of min-max problems. 
	
	\subsection{Discussion on Possible Alternative Approaches}
	\label{sec:discussion}
	We have shown how to utilize the strict-convexity-strict-concavity to design algorithms with initialization-dependent rates. However, one may wonder about alternative approaches to achieve the same goal. For example, adding arbitrary small amount of $\ell_2$ regularization to both primal and dual variables will induce strict-convexity-strict-concavity with negligible bias. However, note that any arbitrarily small bias would not be negligible when the number of iterations $T$ gets large. 
	
	In alternative, one could think to trade off the bias with the rate. But, this approach seems to be able to recover only a $DG/\sqrt{T}$ rate ($D$ is the domain size, $G$ is the gradient upper bound), so it would not be initialization-dependent. Let us consider the function $F(\x,\y)$ with $\|\x\|\leq D$ and $\|\y\|\leq D$. In particular, denote the primal function by $p(\x)=\max_{\y\in\mathcal{Y}}F(\x,\y)$, the dual function by $d(\y)=\min_{\x\in\mathcal{X}}F(\x,\y)$. Define the regularized primal function by $p_{\mu}(\x)=\max_{\y\in\mathcal{Y}}F(\x,\y)+\mu\|\x\|^2$, and the regularized dual function by $d_{\mu}(\y)=\min_{x\in\mathcal{X}}F(\x,\y)-\mu\|\y\|^2$, where $\mu>0$ is a very small constant. Since the domain is bounded (assume both $\mathcal{X}$ and $\mathcal{Y}$ have diameter $D$), we have $|d(\y)-d_{\mu}(\y)|\leq \mu D^2$ and $|p(\x)-p_{\mu}(\x)|\leq \mu D^2$ for any $\x\in\mathcal{X}$ and any $\y\in\mathcal{Y}$. By using the state-of-the-art solver for the strongly-convex-strongly-concave subproblem, we have $p_{\mu}(\x_T)-d_{\mu}(\y_T)\leq \frac{G^2}{\mu T}$, and hence for the original problem, the duality gap is $p(\x_T)-d(\y_T)\leq \frac{G^2}{\mu T} + \mu D^2$. Choosing the optimal $\mu=\frac{G}{D\sqrt{T}}$, the rate becomes $DG/\sqrt{T}$. Please note that the error term $\mu D^2$ explicitly depends on the size of the domain, so it seems difficult to get around the domain diameter by the approach of adding a small $\ell_2$ regularization.
	In contrast, directly assuming strict-convexity-strict-concavity we can get initialization-dependent asymptotic rates.
	
	\section{Fast Rates under Growth Condition and H{\"o}lder Continuous Solution Mapping}
	\label{sec:extension1}
	
	In this section, we consider an extension of Algorithm~\ref{alg:cb-minmax} when the function satisfies a growth condition and a H{\"o}lder continuous solution mapping, in which we establish improved rates.

	
	
	
	%
	
	
	We consider the following assumptions.
	\begin{ass}[Growth condition]
		\label{ass:3}
		(i) $\|\x-\x_*\|\leq c_1(F(\x,\y_*)-F(\x_*,\y_*))^{\theta_1}$ and $\|\y-\y_*\|\leq c_2(F(\x_*,\y_*)-F(\x_*,\y))^{\theta_1}$, where $c_1>0$, $c_2>0$, $0< \theta_1\leq 1$, $(\x_*,\y_*)$ is the optimal solution of the original problem~(\ref{eq:problem}).
	\end{ass}
	
	\begin{ass}[H{\"o}lder continuous solution mapping]
		\label{ass:4}
		Define $H_1(\y)=\arg\min_{\x\in\mathcal{X}} F(\x,\y)$, $H_2(\x)=\arg\max_{\y\in\mathcal{Y}}F(\x,\y)$.  $\left\|H_1(\y_1)-H_1(\y_2)\right\|\leq L_{\y}\left\|\y_1-\y_2\right\|^{\theta_2}$ for any $\y_1,\y_2\in\mathcal{Y}$, and $\left\|H_2(\x_1)-H_2(\x_2)\right\|\leq L_{\x}\left\|\x_1-\x_2\right\|^{\theta_2}$ for any $\x_1,\x_2\in\mathcal{X}$, where $0<\theta_2\leq 1$.
	\end{ass}
	
	\textbf{Remark}: Assumption~\ref{ass:3} is a generalization of growth condition in minimization problems~\citep{li2013global,yang2018rsg}. Assumption~\ref{ass:4} is closely related to the Aubin Property~\citep{aubin1984lipschitz}, which is usually employed to characterize the Lipschitz behavior of solution set for convex optimization problems. Examples satisfying this assumption can be found in~\citep{dontchev2009implicit} (e.g., Example 3B.6, Exercise 3C.5). We want to emphasize that it covers a wide range of min-max problems. For example, strongly-convex-strongly-concave min-max problems satisfy Assumption~\ref{ass:3} and~\ref{ass:4} with $\theta_1=1/2$ and $\theta_2=1$ ($\theta_1=1/2$ is due to the definition of strong-convexity-strong-concavity, $\theta_2=1$ is shown in Lemma 2.2 of~\cite{ghadimi2018approximation}). Other examples satisfying Assumption~\ref{ass:2}, ~\ref{ass:3} and~\ref{ass:4} can be found in Appendix~\ref{app:examples}.
	
	%
	\begin{algorithm}[t]
		\caption{Restart-CB-Min-Max}
		\label{alg:cb-minmaxrestart}
		\begin{algorithmic}[1]
			\REQUIRE $\widehat{\x}_0$, $\widehat{\y}_0$
			\FOR{$s=1,\dots,S$}
			\STATE $(\widehat{\x}_s, \widehat{\y}_s)=\text{CB-Min-Max}( \widehat{\x}_{s-1},\widehat{\y}_{s-1}, T_s)$
			\ENDFOR
			\STATE \textbf{Return} $\widehat{\x}_S, \widehat{\y}_S$
		\end{algorithmic}
	\end{algorithm}
	

	Under Assumption~\ref{ass:3} and~\ref{ass:4}, we can design a restart version of the algorithm, which  is presented in Algorithm~\ref{alg:cb-minmaxrestart}. We will prove that Algorithm~\ref{alg:cb-minmaxrestart} enjoys faster non-asymptotic convergence rates. 
	
	We first provide convergence guarantee for one-stage of Algorithm~\ref{alg:cb-minmaxrestart}, which is presented in Theorem~\ref{fastrate:thm1} (the proof is included in Appendix~\ref{appendix:fastrate}).
	\begin{thm}
		\label{fastrate:thm1}
		Suppose Assumptions~\ref{ass:2}, \ref{ass:3} and~\ref{ass:4} hold. Running Algorithm~\ref{alg:cb-minmax} for $T$ iterations yields
		\begin{equation}
			\begin{aligned}
				\max_{\y\in\mathcal{Y}} F(\bar{\x}_T, \y) - \min_{\x\in\mathcal{X}} F(\x,\bar{\y}_T)
				\leq \mathcal{O}\left(\frac{1}{T}+\frac{1}{T^{\frac{\theta+1}{2}}}+\frac{\text{ObjGap}^{\theta_1}(\x_0,\y_0)\ln T}{\sqrt{T}}\right),
			\end{aligned}
		\end{equation}
		where $\text{ObjGap}(\x_0,\y_0)\triangleq F(\x_0,\y_*)-F(\x_*,\y_*)+F(\x_*,\y_*)-F(\x_*,\y_0)$, and $\theta\triangleq\theta_1\theta_2$.
	\end{thm}

	Based on Theorem~\ref{fastrate:thm1}, we then introduce the Theorem~\ref{fastrate:thm2}, which illustrates the improved rate achieved by Algorithm~\ref{alg:cb-minmaxrestart} when $\theta_1>0$ and $\theta_2>0$. The proof of Theorem~\ref{fastrate:thm2} is included in Appendix~\ref{appendix:fastrate}.
	\begin{thm}
		\label{fastrate:thm2}
		Suppose Assumptions~\ref{ass:2}, \ref{ass:3} and~\ref{ass:4} hold.  Define $\theta\triangleq\theta_1\theta_2$. Assume $\text{ObjGap}(\widehat{\x}_0,\widehat{\y}_0)\triangleq F(\widehat{\x}_0,\y_*)-F(\x_*,\y_*)+F(\x_*,\hat{\y}_0)-F(\x_*,\y_*)\leq \epsilon_0$. Define $\epsilon_s=\epsilon_0/2^s$. Run Algorithm~\ref{alg:cb-minmaxrestart} for $S=\lfloor\log\left(\epsilon_0/\epsilon\right)\rfloor$ stages with $T_s=\widetilde{\mathcal{O}}(\max(\frac{1}{\epsilon_s^{2-2\theta_1}},\frac{1}{\epsilon_s^{2/(1+\theta)}}))$. Then, we can guarantee that $\max_{\y\in\mathcal{Y}} F(\widehat{\x}_S, \y) - \min_{\x\in\mathcal{X}} F(\x,\widehat{\y}_S)\leq\epsilon$ and the total iteration complexity is  $\widetilde{\mathcal{O}}\left(\max\left(\frac{1}{\epsilon^{2(1-\theta_1)}}, \frac{1}{\epsilon^{2/(1+\theta)}}\right)\right)$.
	\end{thm}
	
	\paragraph{Remark} When the function is strongly-convex-strongly-concave, the state-of-the-art complexity is $O(1/\epsilon)$ while the complexity of our Algorithm~\ref{alg:cb-minmaxrestart} is $\widetilde{O}(1/\epsilon^{4/3})$ ($\theta_1=1/2$, $\theta_2=1$, $\theta=\theta_1\theta_2=1/2$), and hence our algorithm does not achieve state-of-the-art in this particular case. However, our analysis captures a more general class of min-max problems, and the complexity is strictly better than $O(1/\epsilon^2)$ as long as $\theta_1>0$ and $\theta_2>0$. Intuitively speaking, our algorithm can take advantage of the growth condition of the loss landscape and enjoy faster convergence than the standard algorithm (e.g., GDA) which is oblivious of such favorable structure. To the best of our knowledge, leveraging function growth condition for min-max problems is novel and does not appear in the previous literature.
	
	\begin{figure}[t]
		\centering
		\subfigure[Initialization: $(0.1, 0.1)$]{
			\includegraphics[scale=0.3]{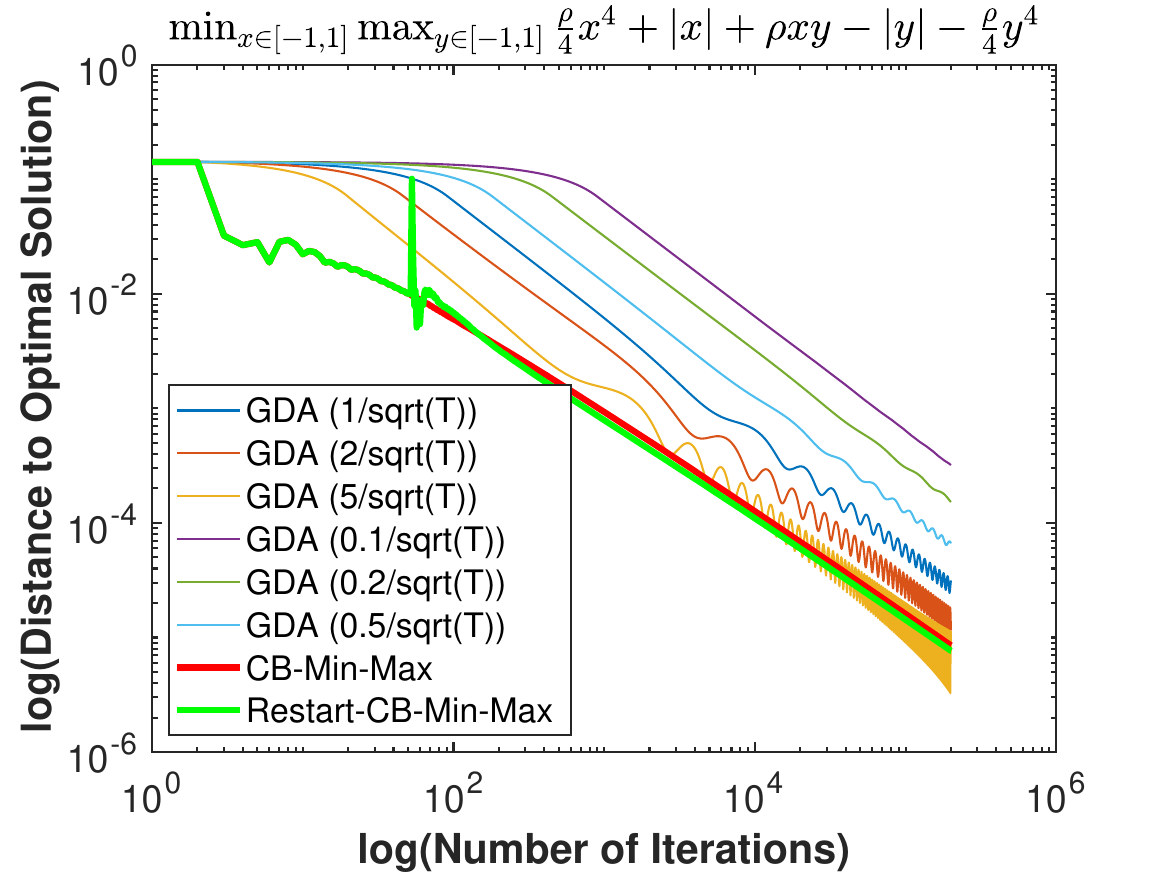}
		}
		\subfigure[Initialization: $(0.05. 0.05)$]{
			\includegraphics[scale=0.3]{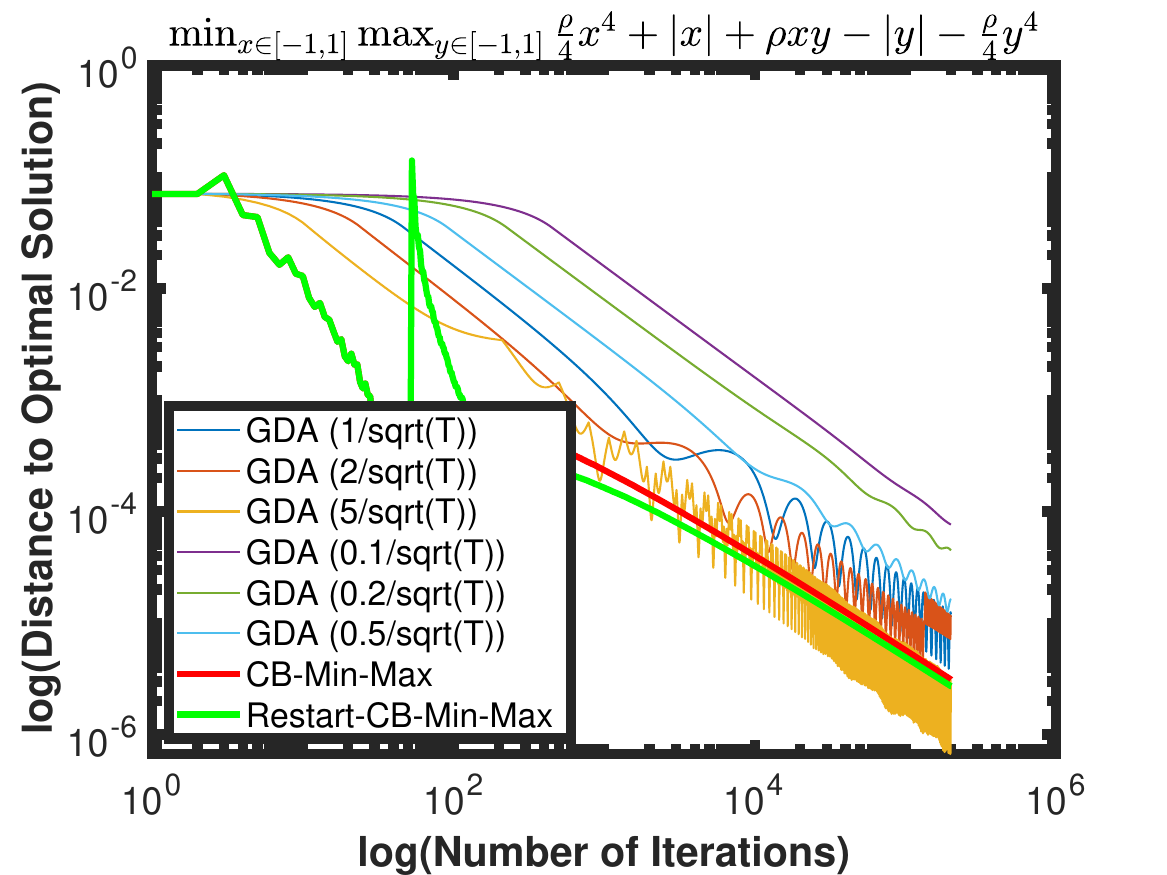}
		}
		\caption{Comparison of different algorithms for the synthetic problem~(\ref{prob:synthetic}). GDA($\cdot$) stands for stands for gradient descent ascent (with learning rate), where $T$ is total number of iterations. CB-Min-Max stands for Algorithm~\ref{alg:cb-minmax}. Restart-CB-Min-Max stands for Algorithm~\ref{alg:cb-minmaxrestart}.} 
		\label{fig:synthetic}
	\end{figure}
	\section{Experiments}
	\label{sec:experiment}
	In this section, we conduct experiments to justify the effectiveness of our proposed algorithm. 
	We first consider a synthetic problem. Another problem is distributionally robust optimization (DRO), which requires an algorithm which can deal with non-Euclidean space such as probability simplex. For lack of space, the algorithm, theory and experiments about DRO are included in Appendix~\ref{appendix:extension2}.
	
	\noindent\textbf{Synthetic Problem} We consider the following synthetic min-max problem:
	\begin{equation}
		\label{prob:synthetic}
		\min_{x\in\mathcal{X}}\max_{y\in\mathcal{Y}}F(x,y) := \frac{\rho}{4}x^4+|x|+\rho xy-|y|-\frac{\rho}{4}y^4,
	\end{equation}
	where $\mathcal{X}=\{x\,|\, |x|\leq R_x\}$ and $\mathcal{Y}=\{y\,|\, |y|\leq R_y\}$, $\rho, R_x, R_y$ are all positive constants. One can show that this problem satisfies Assumptions~\ref{ass:2},~\ref{ass:3},~\ref{ass:4} simultaneously with $\theta_1=1$ and $\theta_2=1/3$ (the proof of this claim is included in Appendix~\ref{app:examples}). 
	In addition, the optimal solution of~(\ref{prob:synthetic}) is $(0,0)$.

	In our experiment, we set $R_x=R_y=1$, $\rho=0.5$. We consider two different initializations $(x_0,y_0)=(0.1,0.1)$ and $(x_0,y_0)=(0.05,0.05)$ to test our algorithms and other baselines.
	We choose the same initial point for both primal-dual gradient method with and CB-Min-Max, and report the distance to the optimal solution versus the number of iterations. The learning rate of the gradient descent ascent method (GDA) is set to be $\frac{c}{G\sqrt{T}}$, where $c$ is tuned from $\{1,2,5,0.1,0.2,0.5\}$, $G$ is the gradient's upper bound and $T$ is the number of iterations ($T=2\times 10^5$). For CB-Min-Max and Restart-CB-Min-Max, we set all gradients in Algorithm~\ref{alg:cb-minmax} to be scaled by its upper bound $G$ to make sure that the scaled gradient has norm smaller than $1$. For Restart-CB-Min-Max, the algorithm restarts at the $ar^{s}$-th iteration, where $a=50$, $r=10^4$, $s=0,1,\ldots$. We plot the loglog curves for distance to the optimal solution versus the number of iterations for two different initializations, which are presented in Figure~\ref{fig:synthetic} (a) and (b).
	
	From Figure~\ref{fig:synthetic}, we can see that the algorithms CB-Min-Max and Restart-CB-Min-Max are better than the GDA with theoretically optimal learning rate $(2/\sqrt{T})$, and are comparable to the GDA with best-tuned learning rate $(5/\sqrt{T})$. Note that $2/\sqrt{T}$ learning rate for GDA is the theoretically best since the domain's size is $D=2$, while the $5/\sqrt{T}$ learning rate gives the best performance as shown in Figure~\ref{fig:synthetic}. Another interesting observation is that Restart-CB-Min-Max is slightly better than CB-Min-Max at the very end, and it fluctuates a little bit in the middle. The reason is due to the restart. When the restart happens, the algorithm enters into a new stage and the update becomes a bit aggressive at the very beginning of this stage, and then it quickly converges to a good solution. Finally, when comparing (a) and (b) in Figure~\ref{fig:synthetic}, we can see that our algorithms indeed have better performance when the initialization is better, and this is also consistent with our theory. 
	
	We want to emphasize that the problem instance~(\ref{prob:synthetic}) is 1 dimension and nonsmooth, so GDA with best-tuned learning rate is the strongest baseline in this case. Other popular algorithms are not applicable or not better than the well-tuned GDA. For example, extragradient method~\citep{korpelevich1976extragradient} requires the function to be smooth and hence is not applicable to nonsmooth problems such as~(\ref{prob:synthetic}), and coordinate-wise adaptive gradient algorithm~\citep{bach2019universal} is the same as the nonadaptive version since (\ref{prob:synthetic}) is of dimension 1.
	
	
	\section{Conclusion}
	In this paper, we consider how to get initialization-dependent convergence rate of first-order algorithms for convex-concave min-max problems. We first identify a condition (i.e., strict-convexity-strict-concavity) and show that it is sufficient to get the initialization-dependent rate. We also take advantage of a parameter-free algorithm with this initialization-dependent rate to design a better algorithm with fast non-asymptotic convergence rate, for min-max problems with a growth condition and H\"{o}lder continuous solution mapping. 
	Experimental results show the superior performance of the proposed algorithms. In the future, we plan to study the lower bound for parameter-free algorithms in min-max problems.
	
\acks{This material is based upon work supported by the National Science Foundation under the grants no. 1908111 ``AF: Small:
Collaborative Research: New Representations for Learning Algorithms and Secure Computation'' and no. 2046096 ``CAREER: Parameter-free Optimization Algorithms for Machine Learning''.}
	
	\bibliography{ref}
	\bibliographystyle{plain}
	\appendix
	\appendix
\newpage
\section*{Appendix}
\section{Proofs in Section~\ref{sec:initialization-rate}}
\label{appendix:mainsection}
\subsection{Proof of Theorem~\ref{prop:oco-to-minmax}}
\begin{proof}

Note that $\ell_t(\x)=F(\x,\y_t)$, and $h_t(\y)=-F(\x_t,\y)$. By Jensen's inequality, we have
	\begin{align*}
	F(&\bar{\x}_T, \y) - F(\x,\bar{\y}_T)\leq \frac{1}{T}\sum_{t=1}^{T}F(\x_t,\y)- \frac{1}{T} \sum_{t=1}^{T}F(\x,\y_t)\\
	&=\frac{1}{T}\sum_{t=1}^{T}F(\x_t,\y)-\frac{1}{T}\sum_{t=1}^{T}F(\x_t,\y_t)
 +\frac{1}{T}\sum_{t=1}^{T}F(\x_t,\y_t)-\frac{1}{T} \sum_{t=1}^{T}F(\x,\y_t)\\
	&=\frac{1}{T} \sum_{t=1}^{T} (h_t(\y_t)-h_t(\y))+\frac{1}{T}\sum_{t=1}^{T}(\ell_t(\x_t)-\ell_t(\x))~.
\end{align*}
Taking $\x=\x'_T\in\arg\min_{\x\in\mathcal{X}} F(\x,\bar{\y}_T)$ and $\y=\y'_T\in\arg\max_{\y\in\mathcal{Y}} F(\bar{\x}_T,\y)$, we get the stated result.
\end{proof}

\subsection{Proof of Theorem \ref{thm:constrained_new_algo}}
\begin{proof}
Define $\widetilde{\ell}_t(\x)=\frac{1}{2}\left(\left\langle\widehat{\g}_t,\x\right\rangle+\|\widehat{\g}_t\|\|\x-\Pi_{\mathcal{X}}(\x)\|\right)$. Noting that $\g_t\in\partial\widetilde{\ell}_t(\widetilde{\x}_t)$, so Algorithm~\ref{alg:constrained-cb} is a coin-betting algorithm for minimizing regret defined by the function $\widetilde{\ell}$ over the sequence of points $\{\widetilde{\x}_t\}$. By Corollary 5 in~\cite{CutkoskyO18}, we know that
\begin{equation*}
\sum_{t=1}^{T}\left(\widetilde{\ell}_t(\widetilde{\x}_t)-\widetilde{\ell}_t(\u)\right)\leq 1+\|\x_0-\u\|\sqrt{T\ln(24T^2\|\x_0-\u\|^2+1)}~.
\end{equation*}
By the black-box reduction argument (Theorem 3 in~\cite{CutkoskyO18}), we have
\begin{equation*}
\sum_{t=1}^{T}\left(\ell_t(\x_t)-\ell_t(\u)\right)\leq	2\left(\sum_{t=1}^{T}\left(\widetilde{\ell}_t(\widetilde{\x}_t)-\widetilde{\ell}_t(\u)\right)\right)\leq 2+2\|\x_0-\u\|\sqrt{T\ln(24T^2\|\x_0-\u\|^2+1)},
\end{equation*}
that completes the proof.
\end{proof}

\subsection{Proof of Corollary~\ref{cor:1}}
\begin{proof}
Using Theorem~\ref{thm:constrained_new_algo} twice (applying it with $\ell_t=F(\cdot,\y_t)$ and $\u=\x'_T\in\arg\min_{\x\in\mathcal{X}}F(\x,\bar{\y}_T)$, and applying it with $h_t=-F(\x_t,\cdot)$ with $\u=\y'_T\in\arg\max_{\y\in\mathcal{Y}}F(\bar{\x}_T,\y)$) and Theorem~\ref{prop:oco-to-minmax} concludes the proof.
\end{proof}

\subsection{Proof of Theorem~\ref{thm:main}}
\begin{proof}
	We can see that~(\ref{eq:cb1}) satisfies the premises of Theorem~\ref{main:metathm} with $A(x)=x$, $B(T)=\frac{\sqrt{4\ln(24 T^2 D^2+1)}}{\sqrt{T}}$, $C(T)=4/\sqrt{T}$. Note that $A(x+y)\leq A(x)+A(y)$ for any $0\leq x\leq D, 0\leq y\leq D$, then invoking Theorem~\ref{main:metathm}, we have
	\begin{equation*}
		\begin{aligned}
			\max_{\y\in\mathcal{Y}}F(\bar{\x}_T,\y)-\min_{\x\in\mathcal{X}}F(\x,\bar{\y}_T) 
			& \leq \left[A(\|\x_0-\x_*\|)B(T)+ A(\|\y_0-\y_*\|)B(T)\right]+R(T)+\frac{4}{T}\\
			& \leq \frac{2\left(\|\x_0-\x_*\|+\|\y-\y_*\|\right)\sqrt{\ln(24 T^2 D^2+1)}}{\sqrt{T}}+o\left(\frac{1}{\sqrt{T}}\right).
		\end{aligned}
	\end{equation*}
\end{proof}

\section{Proof in Section~\ref{sec:extension1}}
\label{appendix:fastrate}
\subsection{Proof of Theorem~\ref{fastrate:thm1}}

\begin{proof}
	By Corollary~\ref{cor:1}, we have
	\begin{equation}
		\label{growth:eq1}
		\begin{aligned}
			\max_{\y\in\mathcal{Y}} &F(\bar{\x}_T, \y) - \min_{\x\in\mathcal{X}} F(\x,\bar{\y}_T)\\
			&\leq \frac{4}{T}+\frac{\|\x_0-\x'_T\|\sqrt{\ln\left(1+24\|\x_0-\x'_T\|^2T^2\right)}}{\sqrt{T}}+\frac{\|\y_0-\y'_T\|\sqrt{\ln\left(1+24\|\y_0-\y'_T\|^2T^2\right)}}{\sqrt{T}},
		\end{aligned}
	\end{equation}
	where 
	$\x'_T=\arg\min_{\x\in\mathcal{X}}F(\x,\bar{\y}_T)$, $\y'_T=\arg\min_{\y\in\mathcal{Y}}F(\bar{\x}_T,\y)$.
	By the growth condition in Assumption~\ref{ass:3}, we know that 
	\begin{equation}
		\label{growth:eq2}
		\|\x_0-\x_*\|\leq c_1 \left(F(\x_0,\y_*)-F(\x_*,\y_*)\right)^{\theta_1}.
	\end{equation}
	By the H{\"o}lder continuity of the solution mapping (Assumption~\ref{ass:4}), we have $\|\x_*-\x'_T\|=\|\arg\min_{\x} F(\x,\y_*)-\arg\min_{\x} F(\x,\bar{\y}_T)\|\leq L_{\y}\left\|\bar{\y}_T-\y_*\right\|^{\theta_2}$. Note that we have $\left\|\bar{\y}_T-\y_*\right\|\leq c_2(F(\x_*,\y_*)-F(\x_*,\bar{\y}_T))^{\theta_1}$, so we have
	\begin{equation}
		\label{growth:eq3}
		\begin{aligned}
			\|\x_*-\x'_T\|
			&\leq c_2^{\theta_2}L_{\y}\left(F(\x_*,\y_*)-F(\x_*,\bar{\y}_T)\right)^{\theta_1\theta_2}
			\leq c_{\y}\left(F(\x_*,\y_*)-\min_{\x}F(\x,\bar{\y}_T)\right)^{\theta} \\
			&\leq O\left(\frac{c_{\y}D_{\mathcal{Y}}^{\theta}\ln^{\frac{\theta}{2}}(D_{\mathcal{Y}} T)}{T^{\theta/2}}\right)~.
		\end{aligned}
	\end{equation}
	where $c_{\y}=c_2^{\theta_2}L_{\y}$, $\theta=\theta_1\theta_2$, and the last inequality holds due to the convergence guarantee established in Corollary~\ref{cor:1} .
	
	A parallel argument in terms of $\y$ gives
	\begin{equation}
		\label{growth:eq10}
		\|\y_0-\y_*\|\leq c_2 (F(\x_*,\y_*)-F(\x_*,\y_0))^{\theta_1},
	\end{equation}
	and
	\begin{equation}
		\label{growth:eq11}
		\begin{aligned}
			\|\y_*-\y'_T\|
			&\leq c_1^{\theta_1}L_{\x}\left(F(\bar{\x}_T,\y_*)-F(\x_*,\y_*)\right)^{\theta_1\theta_2}
			\leq c_{\x}\left(F(\bar{\x}_T,\y_*)-F(\x_*,\y_*)\right)^{\theta}\\
			&\leq O\left(\frac{c_{\x}D_{\mathcal{X}}^{\theta}\ln^{\frac{\theta}{2}}(D_{\mathcal{X}} T)}{T^{\theta/2}}\right),
		\end{aligned}
	\end{equation}
	where $c_{\x}=c_{1}^{\theta_1}L_{\x}$. 
	
	Due to triangle inequality and~\eqref{growth:eq2} and~\eqref{growth:eq3}, we have
	\begin{equation}
		\|\x_0-\x’_T\| \leq \|\x_0-\x_*\| + \|\x_*-\x’_T\| \leq c_1( F(\x_0,\y_*)-F(\x_*,\y_*) )^{\theta_1}+O(\ln T/T^{\theta/2}).
	\end{equation}
Similarly, we have
\begin{equation}
\|\y_0-\y’_T\| \leq \|\y_0-\y_*\| + \|\y_*-\y’_T\| \leq c_2( F(\x_*,\y_*) - F(\x_*,\y_0) )^{\theta_1}+O(\ln T/T^{\theta/2}).
\end{equation}
Hence, we have
\begin{equation}
	\label{eq:new}
	\|\x_0-\x’_T\| +\|\y_0-\y’_T\| \leq O\left(\text{ObjGap}^{\theta_1}(\x_0,\y_0)\right)+O(\ln T/T^{\theta/2}).
\end{equation}
	Combining~(\ref{growth:eq1}) with~\eqref{eq:new}, we have
	
	\begin{equation}
		\begin{aligned}
			\max_{\y\in\mathcal{Y}} F(\bar{\x}_T, \y) - \min_{\x\in\mathcal{X}} F(\x,\bar{\y}_T)
			\leq O\left(\frac{1}{T}+\frac{\ln T}{T^{\frac{\theta+1}{2}}}+\frac{\text{ObjGap}^{\theta_1}(\x_0,\y_0)\ln T}{\sqrt{T}}\right),
		\end{aligned}
	\end{equation}
	where $\text{ObjGap}(\x_0,\y_0)=F(\x_0,\y_*)-F(\x_*,\y_*)+F(\x_*,\y_*)-F(\x_*,\y_0)$.
\end{proof}

\subsection{Proof of Theorem~\ref{fastrate:thm2}}
\begin{proof}
	Define $\text{DualityGap}(\widehat{\x}_0,\widehat{\y}_0)=\max_{\y\in\mathcal{Y}} F(\widehat{\x}_0, \y)-\min_{\x\in\mathcal{X}}F(\x,\widehat{\y}_0)$. We know that $\text{ObjGap}(\widehat{\x}_0,\hat{\y}_0)\leq \text{DualityGap}(\widehat{\x}_0,\widehat{\y}_0)$. By invoking the subroutine Algorithm~\ref{alg:cb-minmax} to run $T_0=\widetilde{O}\left(\max\left(\frac{1}{\epsilon_0^{2-2\theta_1}},\frac{1}{\epsilon_0^{2/(1+\theta)}}\right)\right)$ iterations, we know that the duality gap at the new point will be decreased to $\epsilon_1=\epsilon_0/2$. Then the Algorithm~\ref{alg:cb-minmaxrestart} restarts by setting the new point as the initial point, and then invokes the subroutine Algorithm~\ref{alg:cb-minmax} to run $T_1=\widetilde{O}\left(\max\left(\frac{1}{\epsilon_1^{2-2\theta_1}},\frac{1}{\epsilon_1^{2/(1+\theta)}}\right)\right)$ number of iterations, and then it restarts again. Algorithm~\ref{alg:cb-minmaxrestart} repeats this process under it reaches $\epsilon$-duality gap. We know that we have $S=\lfloor\log\left(\epsilon_0/\epsilon\right)\rfloor$ stages and hence the total complexity is $\sum_{s=0}^{S}\widetilde{O}\left(\max\left(\frac{1}{\epsilon_s^{2-2\theta_1}},\frac{1}{\epsilon_s^{2/(1+\theta)}}\right)\right)=\widetilde{O}\left(\max\left(\frac{1}{\epsilon^{2(1-\theta_1)}},\frac{1}{\epsilon^{2/(1+\theta)}}\right)\right)$.
\end{proof}

\section{Non-Euclidean Space: Simplex Setup}
\label{appendix:extension2}

Min-max optimization in Non-Euclidean spaces is an important topic, which has broad applications in machine learning (e.g., distributionally robust optimization). In this section, we focus on the simplex setup by considering the following problem
\[
\min_{\x\in\X}\max_{\p\in\Delta_n} \ F(\x,\p),
\]
where 
\[
\Delta_n = \left\{(p_1,p_2,\dots,p_n)\in\R^n \,\vert\, 0\leq p_i\leq 1, \sum_{i=1}^{n}p_i=1\right\}~.
\]

Our algorithm design shares the similar spirit as in the Euclidean case (Algorithm~\ref{alg:cb-minmax}), in which we split the original problem into the primal and dual problems, and employ corresponding coin-betting algorithms for solving them simultaneously. The algorithm is presented in Algorithm~\ref{alg:constrained-cb-simplex-min-max}, which is an synthesis of Algorithm~\ref{alg:constrained-cb} and Algorithm~\ref{alg:constrained-cb-simplex}. It is worth mentioning that Algorithm~\ref{alg:constrained-cb-simplex} is adapted from Algorithm 2 in~\cite{OrabonaP16}, which is a parameter-free coin-betting algorithm in the probability simplex.

The main difficulty of analyzing parameter-free algorithms for min-max problems in a non-Euclidean space is that we have to use a different distance-generating function, and generalize the proofs done in Euclidean space in Section~\ref{sec:initialization-rate} to non-Euclidean spaces. 

We first present two key Lemmas (Lemma~\ref{lem:trianglesimplex} and Lemma~\ref{lem:regretsimplex})) which are useful for our analysis. The proofs are included in Appendix~\ref{appendix:fast-rate-noneuclidean}. Lemma~\ref{lem:trianglesimplex} plays the role of the triangle inequality for the KL-divergence, and Lemma~\ref{lem:regretsimplex} provides the non-Euclidean variant of Theorem~\ref{thm:constrained_new_algo}.

\begin{algorithm}[t]
	\caption{Simplex-Constrained Coin-betting OCO}
	\label{alg:constrained-cb-simplex}
	\begin{algorithmic}[1]
		\REQUIRE Simplex $\Delta_n\in\R^n$, prior distribution $\p_0\in\Delta_n$
		\FOR{$t=1,\dots,T$}
		\STATE $\w_{t,i}=\frac{\sum_{j=1}^{t-1}\widetilde{\g}_{j,i}}{t}(1+\sum_{j=1}^{t-1}\widetilde{\g}_{j,i}\w_{j,i})$, $i=1,\dots,n$
		\STATE $\widehat{\p}_{t,i}=\p_{0,i}\max(\w_{t,i},0)$, $i=1,\dots,n$
		\STATE $\p_t=\frac{\widehat{\p_t}}{\|\widehat{\p}_t\|_1}$ if $\|\widehat{\p}_t\|_1\neq 0$ else $\p_t=\p_0$
		\STATE Receive subgradient $\widehat{\g}_t^{\p}$
		\STATE $\widetilde{\g}_{t,i}=\g_{t,i}^{\p}-\langle\g_t^\p,\p_t\rangle$ if $\w_{t,i}>0$ else $\widetilde{\g}_{t,i}=\max(\g_{t,i}^{\p}-\langle\g_t^\p,\p_t\rangle, 0)$, $i=1,\dots,n$
		\ENDFOR
		\STATE \textbf{Return} $\bar{\p}_T=\frac{1}{T}\sum_{t=1}^{T}\p_t$
	\end{algorithmic}
\end{algorithm}

\begin{algorithm}[t]
	\caption{CB-Min-Max-Simplex}
	\label{alg:constrained-cb-simplex-min-max}
	\begin{algorithmic}[1]
		\REQUIRE $\x_0\in\mathcal{X}$, prior distribution $\p_0\in\Delta_n$
		\STATE $\widetilde{\x}_0=\x_0$
		\FOR{$t=0,\dots,T$}
		\STATE $\x_t = \Pi_{\mathcal{X}}(\widetilde{\x}_t)$
		\STATE Receive Subgradient $\widehat{\g}_t^{\x}$
		\STATE $\g_t^{\x}=\frac{1}{2}\left(\widehat{\g}_t^{\x}+\|\widehat{\g}_t^{\x}\|\cdot \frac{\widetilde{\x}_t-\x_t}{\|\widetilde{\x}_t-\x_t\|}\right)$ (Define $\mathbf{0}/0=\mathbf{0}$)
		\STATE $\w_{t,i}=\frac{\sum_{j=1}^{t-1}\widetilde{\g}_{j,i}}{t}(1+\sum_{j=1}^{t-1}\widetilde{\g}_{j,i}\w_{j,i})$, $i=1,\dots,n$
		\STATE $\widehat{\p}_{t,i}=\p_{0,i}\max(\w_{t,i},0)$, $i=1,\dots,n$
		\STATE $\p_t=\frac{\widehat{\p_t}}{\|\widehat{\p}_t\|_1}$ if $\|\widehat{\p}_t\|_1\neq 0$ else $\p_t=\p_0$
		\STATE Receive subgradient $\widehat{\g}_t^{\p}$
		\STATE $\widetilde{\g}_{t,i}=\widehat{\g}_{t,i}^{\p}-\langle\widehat{\g}_t^\p,\p_t\rangle$ if $\w_{t,i}>0$ else $\widetilde{\g}_{t,i}=\max(\widehat{\g}_{t,i}^{\p}-\langle\widehat{\g}_t^\p,\p_t\rangle, 0)$, $i=1,\dots,n$
		\STATE $\widetilde{\x}_{t+1}=\x_0-\frac{\sum_{j=1}^{t}\g_{t}^{\x}}{t+1}(1-\sum_{j=1}^{t}\langle \g_{j}^{\x},\widetilde{\x}_{j}\rangle)$
		\ENDFOR
		\STATE \textbf{Return} $\bar{\x}_T=\frac{1}{T}\sum_{t=1}^{T}\x_t$, $\bar{\p}_T=\frac{1}{T}\sum_{t=1}^{T}\p_t$
	\end{algorithmic}
\end{algorithm}

\begin{lem}
	\label{lem:trianglesimplex}
	Assume $\bpi$ is a discrete probability distribution over $\R^n$ with $\pi_i>0$.
	Let $\p, \q$ arbitrary vectors in the probability simplex. 
	Then,
	\begin{equation}
		KL(\p,\bpi) 
		\leq KL(\p, \q)+ \max_i \max\left(\ln \frac{q_i}{\pi_i} ,0\right) \|\p - \q\|_1 
		+ KL(\q,\bpi)~.
	\end{equation}
\end{lem}
%

\begin{lem}
	\label{lem:regretsimplex}
	Define $h_t(\p) = -F(\x_t,\p)$. Then, Algorithm~\ref{alg:constrained-cb-simplex-min-max} guarantees that for any $\p\in\Delta_n$, we have
	\begin{equation}
		\label{eq:parallelp}
		\begin{aligned}
			\frac{1}{T} \sum_{t=1}^{T} h_t(\p_t)-\frac{1}{T}\sum_{t=1}^{T} h_t(\p)
			\leq \frac{1}{T}+\frac{\sqrt{\ln T+3KL(\p,\p_0)}}{\sqrt{T}}~.
		\end{aligned}
	\end{equation}
\end{lem}

Here, we present our main theorem in this section.
\begin{thm}
	\label{thm:simplex}
	Suppose Assumption~\ref{ass:2} holds and $\p_{0,i}>0$ for every $i=1,\dots,n$. Then, Algorithm~\ref{alg:constrained-cb-simplex-min-max} guarantees that 
	\begin{align}
		\max_{\p\in\Delta_n} &F(\bar{\x}_T, \p) - \min_{\x\in\mathcal{X}} F(\x,\bar{\p}_T) \\
		&\leq \frac{4}{T}+\frac{\|\x_0-\x_*\|\sqrt{\ln\left(1+24\|\x_0-\x_*\|T^2\right)}}{\sqrt{T}} 
		+\frac{\sqrt{\ln T+3 KL(\p_*,\p_0)}}{\sqrt{T}}+o\left(\frac{1}{\sqrt{T}}\right)~.
	\end{align}
\end{thm}
The high-level proof idea of Theorem~\ref{thm:simplex} is to extend the proof of Theorem~\ref{thm:main} from the Euclidean setting to the probability simplex. The main technical ingredient is utilizing the expansion of the KL divergence in Lemma~\ref{lem:trianglesimplex} and making use of the objective gap bound in Lemma~\ref{lem:regretsimplex}. The detailed proof can be found in Appendix~\ref{appendix:fast-rate-noneuclidean}. 
%

\section{Proofs in Appendix~\ref{appendix:extension2}}
\label{appendix:fast-rate-noneuclidean}
\subsection{Proof of Lemma~\ref{lem:trianglesimplex}}
\begin{proof}
We have
\begin{equation}
	\begin{aligned}
		KL(\p,\bpi) 
		&= \sum_{i=1}^d p_i \ln \frac{p_i}{\pi_i}
		= \sum_{i=1}^d p_i \ln \frac{p_i}{q_i} + \sum_{i=1}^d p_i \ln\frac{q_i}{\pi_i}\\
		&= KL(\q,\bpi) + \sum_{i=1}^d q_i \ln\frac{q_i}{\pi_i} + \sum_{i=1}^d (p_i-q_i) \ln\frac{q_i}{\pi_i}\\
		&\leq KL(\q,\bpi) + \sum_{i=1}^d q_i \ln\frac{q_i}{\pi_i} + \max_i \max(\ln\frac{q_i}{\pi_i},0) \sum_{i=1}^d |p_i-q_i|.
	\end{aligned}
\end{equation}
\end{proof}
\subsection{Proof of Lemma~\ref{lem:regretsimplex}}
\begin{proof}
	The proof follows Corollary 6 in~\cite{OrabonaP16} and Jensen's inequality in terms of $\p$.
\end{proof}
\subsection{Proof of Theorem~\ref{thm:simplex}}
\begin{proof}
	By combining Lemma~\ref{lem:regretsimplex}, Theorem~\ref{thm:constrained_new_algo} and Lemma~\ref{prop:oco-to-minmax}, we can get the following bound
	\begin{equation}
		\label{eq:simplexbound1}
		\begin{aligned}
			\max_{\p\in\Delta_n} &F(\bar{\x}_T, \p) - \min_{\x\in\mathcal{X}} F(\x,\bar{\p}_T) \\
			&\leq \frac{4}{T}+\frac{\|\x_0-\x'_T\|\sqrt{\ln\left(1+24\|\x_0-\x_*\|T^2\right)}}{\sqrt{T}}  
			+\frac{\sqrt{\ln T+3 KL(\p'_T,\p_0)}}{\sqrt{T}},
		\end{aligned}
	\end{equation}
	where $\x'_T=\arg\min_{\x\in\mathcal{X}}F(\x,\bar{\p}_T)$ and $\p'_T=\arg\max_{\p\in\Delta_n}F(\bar{\x}_T, \p)$.
	Note that \eqref{eq:simplexbound1} is the counterpart of Corollary~\ref{cor:1} in the probability simplex setup.
	By Lemma~\ref{lem:regretsimplex} (taking $\bpi=\p_0$, $\p=\p'_T$ and $\q=\p_*$), we know that 
	\begin{equation}
		\label{eq:simplexbound2}
		\begin{aligned}
			KL(\p'_T,\p_0)\leq	KL(\p'_T,\p_*)
			+\max_{i}\max\left(\ln\frac{\p_{*,i}}{\p_{0,i}},0\right)\left\|\p'_T-\p_*\right\|_1+KL(\p_*,\p_0)~.
		\end{aligned}
	\end{equation}
	Then, we are able to follow the proof of Theorem~\ref{main:metathm} to show that $KL(\p'_T,\p_*)\rightarrow 0$ (this implies that $\max_{i}\max\left(\ln\frac{\p_{*,i}}{\p_{0,i}},0\right)\left\|\p'_T-\p_*\right\|_1\rightarrow 0$ by Pinsker's inequality). In addition, it is proved in Theorem~\ref{main:metathm} that $\x'_T\rightarrow\x_*$. Combining these facts with (\ref{eq:simplexbound2}), the theorem is proved.
\end{proof}

\section{More Experimental Results}
\label{appendix:moreexperiment}
\paragraph{Detailed Experimental Settings} \begin{itemize}
	\item SensIT Vehicle (combined): all the classes with label $2$ and $3$ are regarded as class $-1$, and the class $1$ is regarded as class $1$.
	
	\item  dna: all the classes with label $2$ and $3$ are regarded as class $-1$, and the class $1$ is regarded as class $1$.
	
	\item  gisette: we use the original dataset without any preprocessing.
	
	\item protein: all the classes with label $0$ and $2$ are regarded as class $-1$, and the class $1$ is regarded as class $1$.
		\item letter: all the classes with label $1$ to label $25$ are regarded as class $-1$, and the class $26$ is regarded as class $1$.
	
	\item  mnist: all the classes with label $0$ to label $8$ are regarded as class $-1$, and the class $9$ is regarded as class $1$.
	
	\item  madelon: we use the original dataset without any preprocessing.
	
	\item pendigits: all the classes with label $0$ to label $8$ are regarded as class $-1$, and the class $9$ is regarded as class $1$.
\end{itemize}
\paragraph{Distributionally Robust Optimization}
We consider the following distributionally robust optimization problem:
\begin{equation}
	\label{eq:DRO}
	\min_{\|\w\|\leq R}\max_{\p\in\Delta_n}\sum_{i=1}^{n}p_i\ell_i(\w)+\frac{\lambda}{2}\left\|\p-\frac{\mathbf{1}_n}{n}\right\|^2+\frac{\rho}{2}\|\w\|^2,
\end{equation}
where $\ell_i(\w)=\max(0,1-y_i\w^\top\x_i)$ with $(\x_i,y_i)\in\R^d\times\{-1,+1\}$ being the feature-label pair and $\w\in\R^m$ being the model parameter, $\p$ is a probability vector, $n$ is the number of training examples, $\mathbf{1}_n=[1;1;\ldots;1]$ with length $n$. In our experiment, we set $R=10^5$, $\lambda=\rho=10^{-4}$. Because of the added regularizer, the function becomes strictly-convex-strictly-concave and hence satisfies our Assumption~\ref{ass:2}. We compare two algorithms: the primal-dual gradient~\citep{nemirovski2009robust} and our Algorithm~\ref{alg:constrained-cb-simplex-min-max} (CB-Min-Max-Simplex). In this setup, primal-dual gradient method updates $\w$ by gradient descent and updates $\p$ by exponential gradient ascent simultaneously. The learning rates are set to be $\frac{2R}{G_{\w}\sqrt{T}}$ and $\frac{\log(n)}{G_{\p}\sqrt{T}}$ respectively for primal variable and dual variable, where $G_{\w}$ is the 2-norm of gradient in terms of $\w$, $G_{\p}$ is the infinity-norm of the gradient in terms of $\p$, and $T$ is the number of iterations. Both algorithms start from $\w_0=0$, $\p_0=[1/n;1/n,\ldots,1/n]$ and run $T=1000$ iterations (each iteration amounts to one pass of the training set). We test our algorithms on four benchmark datasets from libsvm website\footnote{https://www.csie.ntu.edu.tw/~cjlin/libsvmtools/datasets/} (SensIT Vehicle (combined), dna, gisette, protein, letter, mnist, madelon, pendigits). We report the training loss and test loss, as shown in the Figure~\ref{fig:svmdataset}.  It can be observed that our proposed parameter-free algorithm (i.e., Algorithm~\ref{alg:constrained-cb-simplex-min-max}) significantly outperforms the standard primal-dual gradient method.
\begin{figure}[t]
	\includegraphics[scale=0.18]{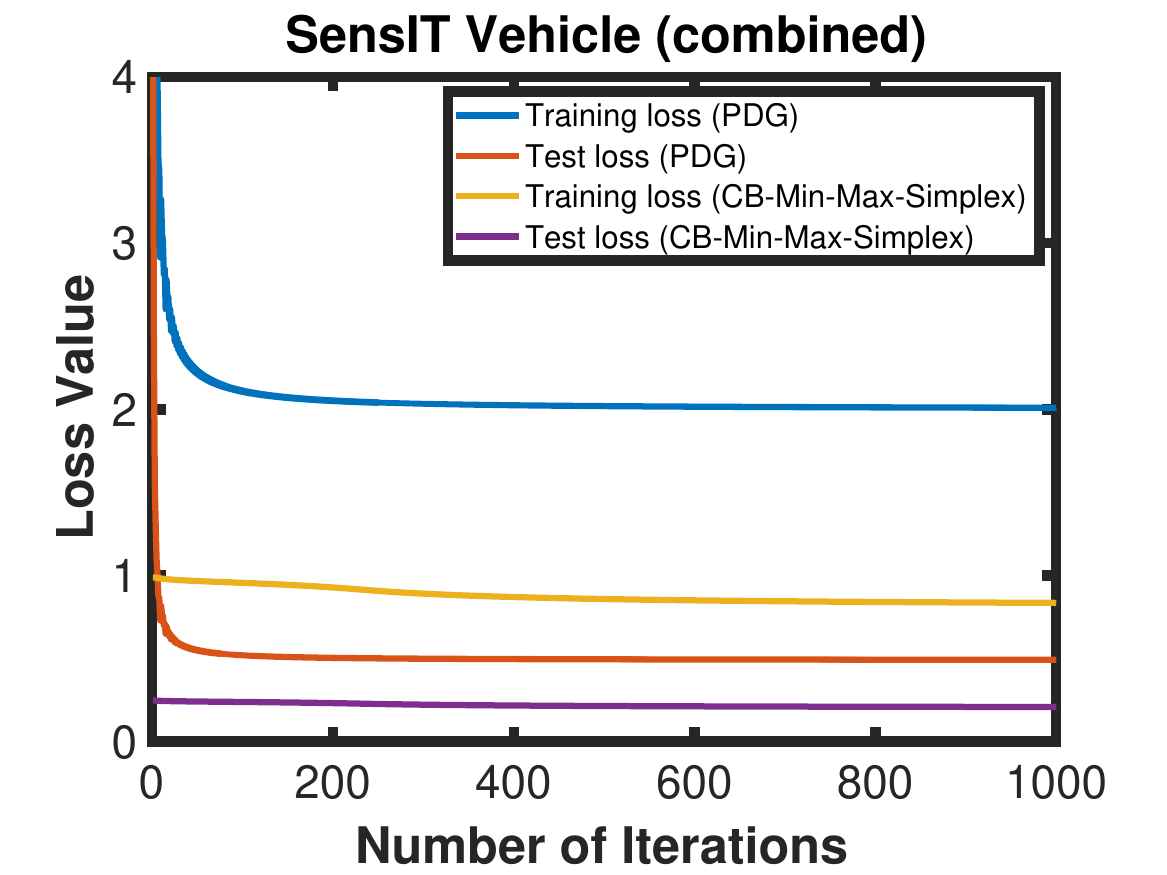}
	\includegraphics[scale=0.18]{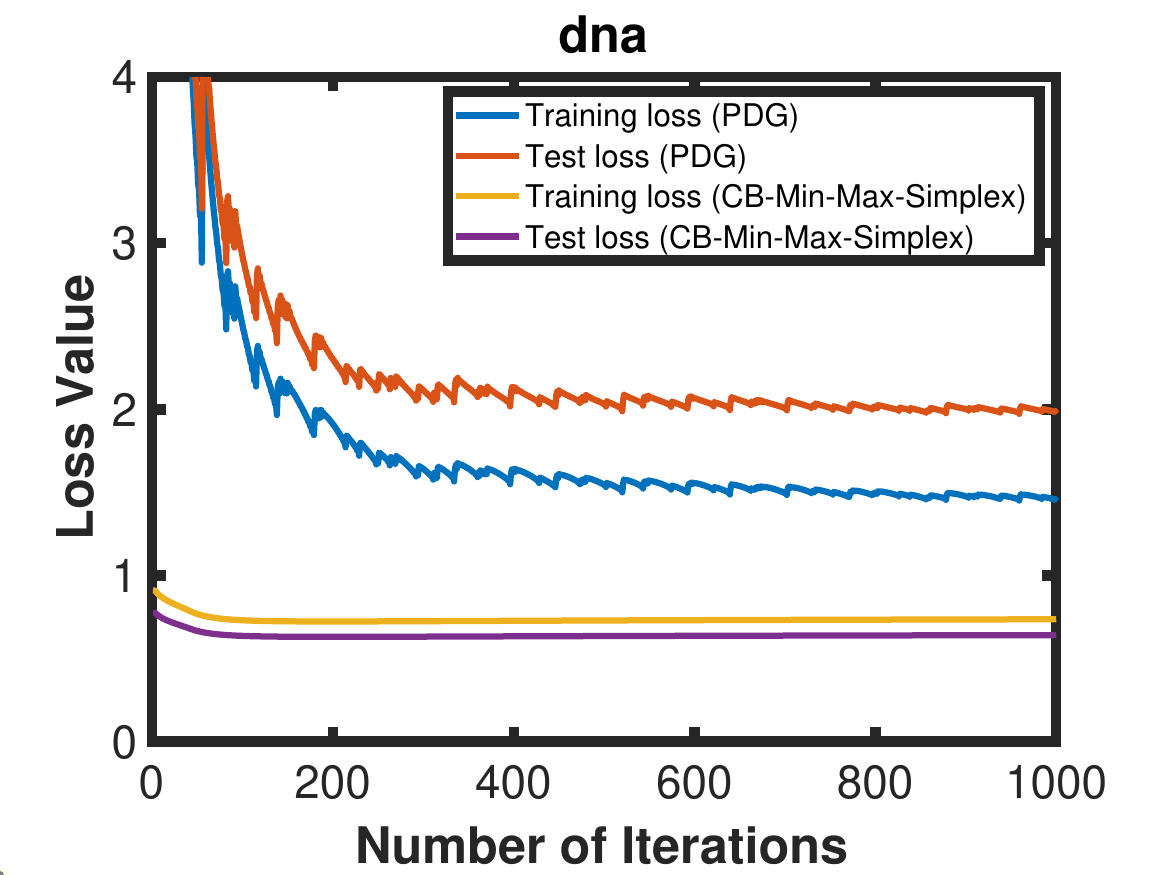}
	\includegraphics[scale=0.18]{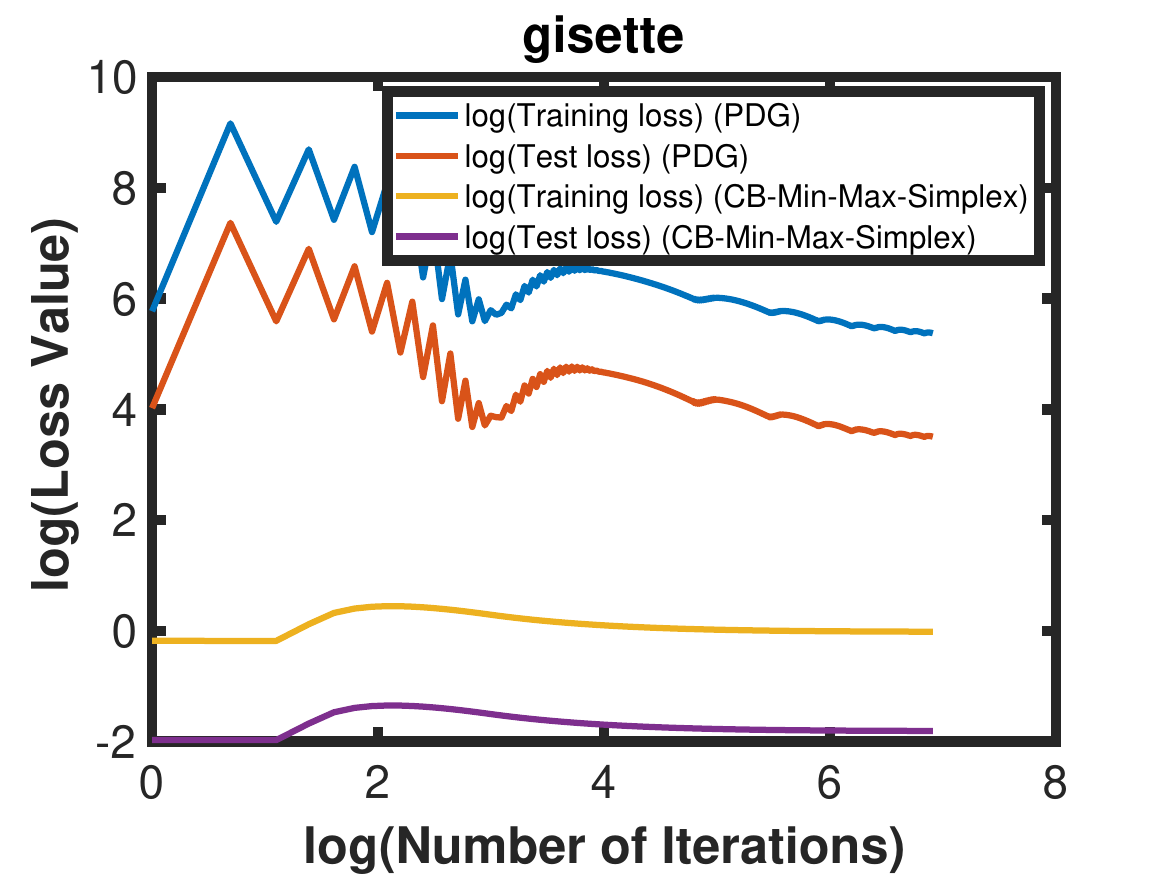}
	\includegraphics[scale=0.18]{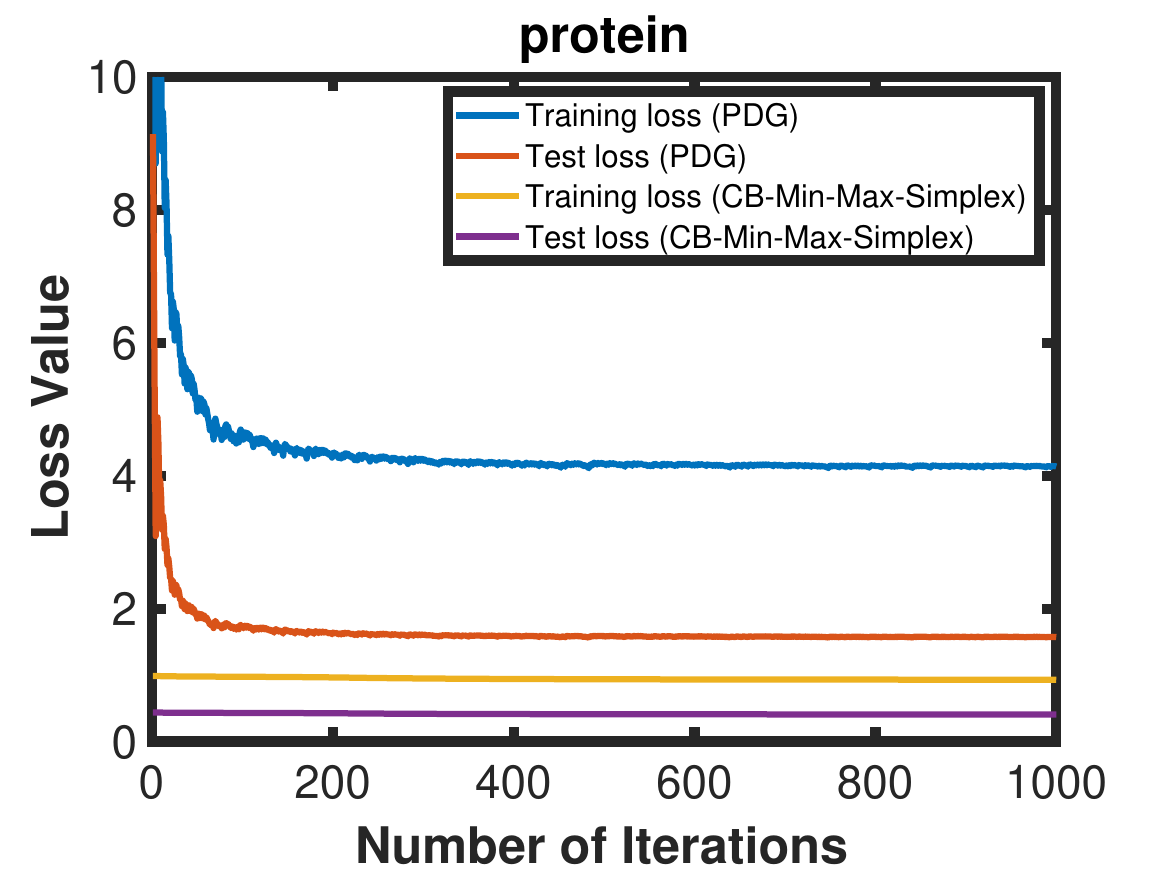}
	\includegraphics[scale=0.18]{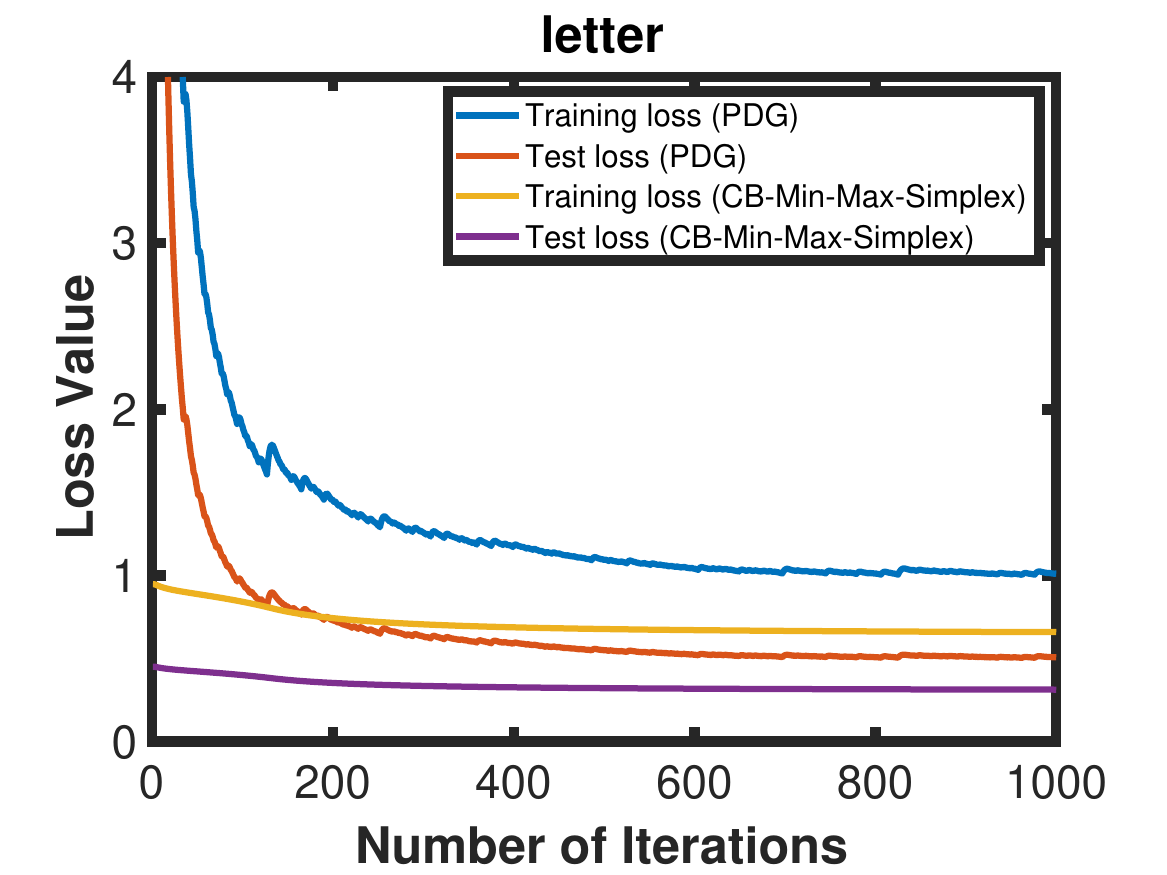}
	\hspace{0.05in}
	\includegraphics[scale=0.18]{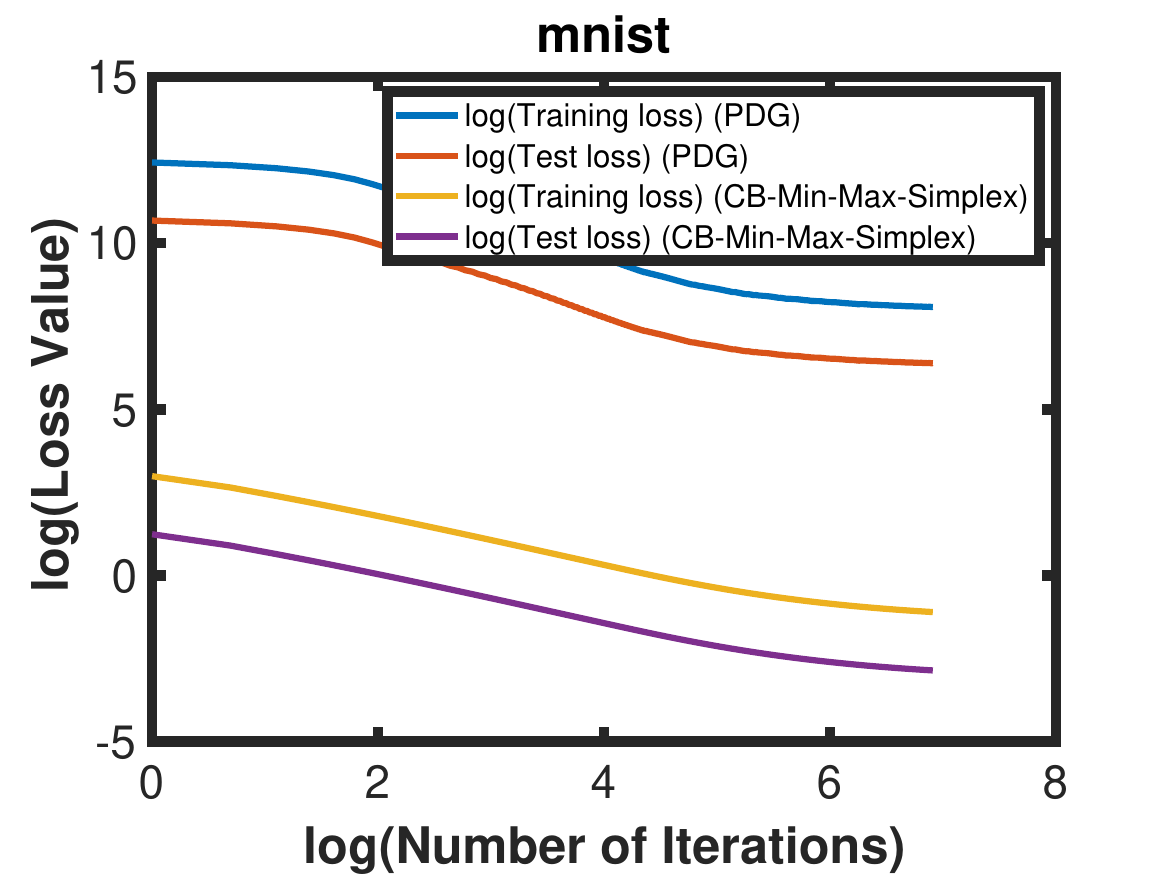}
	\hspace{0.05in}
	\includegraphics[scale=0.18]{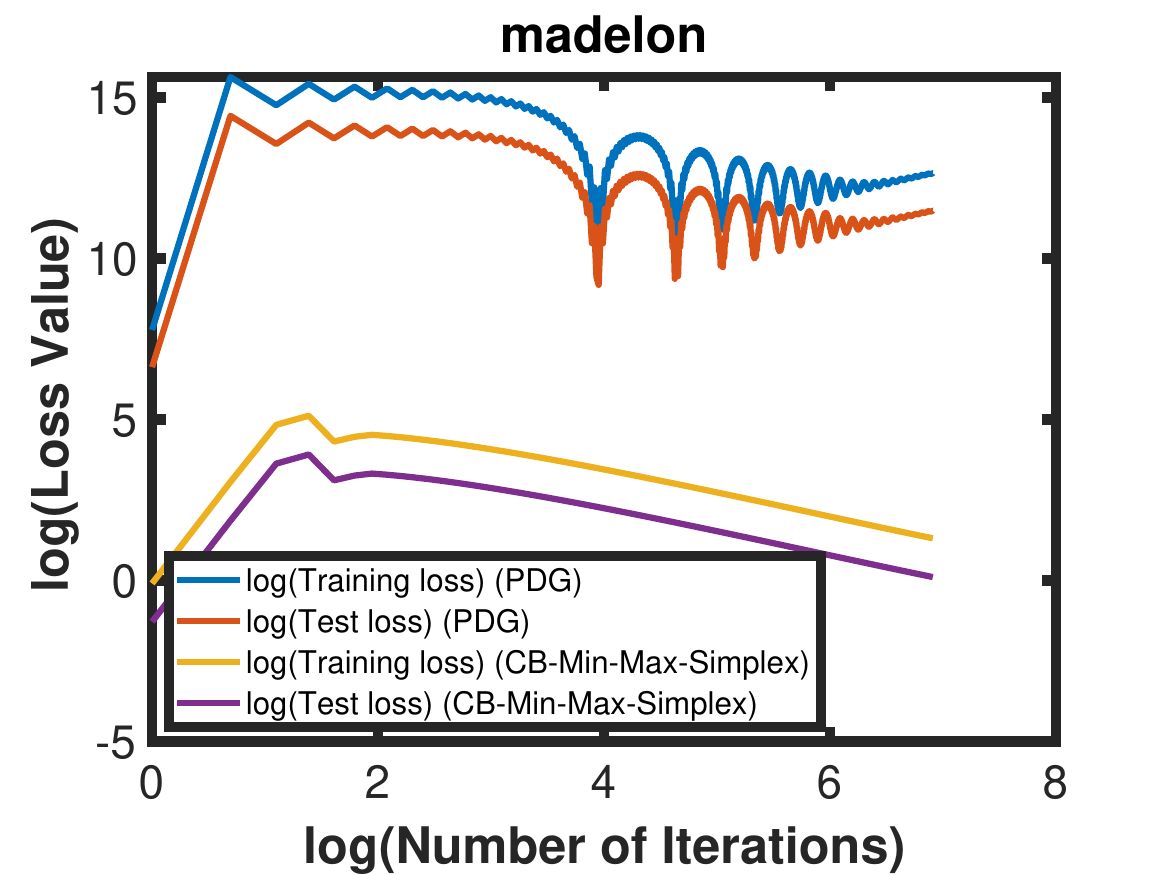}
		\hspace{0.05in}
	\includegraphics[scale=0.18]{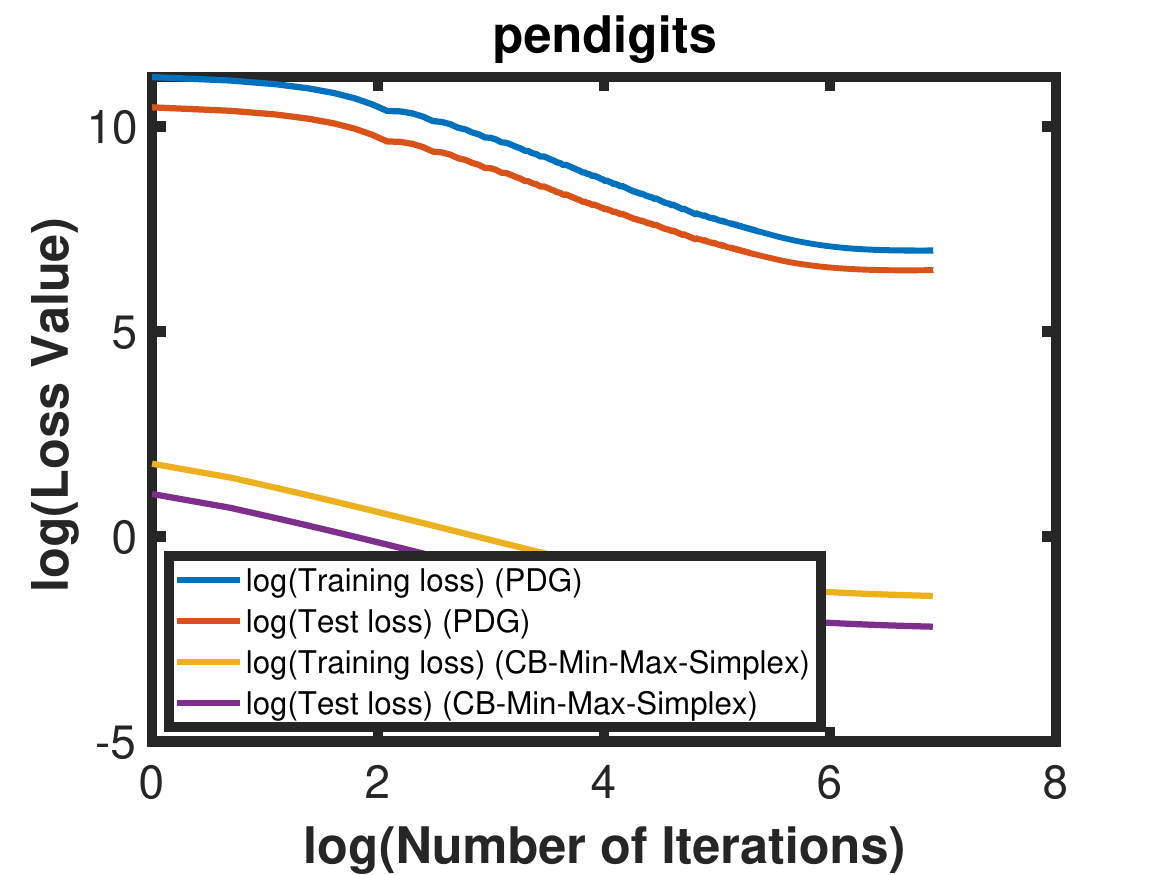}
	\caption{Comparison of different algorithms for the distributionally robust optimization problem~(\ref{eq:DRO}) on SensIT Vehicle (combined), dna, gisette, protein,  letter, mnist, madelon, pendigits  datasets. PDG stands for primal-dual gradient method, CB-Min-Max-Simplex stands for Algorithm~\ref{alg:constrained-cb-simplex-min-max}.}
	\label{fig:svmdataset}
	\vspace*{-0.1in}
\end{figure}

\section{Examples Satisfying Strict-Convexity-Strict-Concavity, Growth Condition and H{\"{o}}lder Continuous Solution Mapping}
\label{app:examples}
\paragraph{Example 1:} We show that problem~(\ref{prob:synthetic}) $F(x,y)=\frac{\rho}{4}x^4+|x|+\rho xy-|y|-\frac{\rho}{4}y^4$, where $|x|\leq R_x$ and $|y|\leq R_y$, satisfies the Assumptions~\ref{ass:2}, \ref{ass:3} and~\ref{ass:4}.
\begin{proof}
	
We first show that the function is strictly-convex-strictly-concave. The function $F(x,y)$ is strictly-convex in $x$, since $\frac{\rho}{4}x^4$ is strictly convex in terms of $x$ and $|x|+\rho xy$ is convex in terms of $x$, and the summation of a strictly convex function and a convex function is strictly convex. The strict-concavity in terms of $\y$ can be proved using a parallel argument.

We then show that Assumption~\ref{ass:3} holds with $\theta_1=1$. We know that the optimal solution of the problem (\ref{prob:synthetic}) is $(x_*,y_*)=(0,0)$, so we have that 
\begin{equation*}
	\begin{cases}
	|x-x_*|=|x|\leq \frac{\rho}{4}x^4+|x|=F(x,y_*)-F(x_*,y_*),\\
		|y-y_*|=|y|\leq \frac{\rho}{4}y^4+|y|=F(x_*,y_*)-F(x_*,y).
\end{cases}
\end{equation*}
and hence we have shown that $\theta_1=1$.

Finally, we show that Assumption~\ref{ass:4} holds with $\theta_2=1/3$. One can show that
\begin{equation*}
	\arg\min_{x}F(x,y)=
	\begin{cases}
0, &\text{ if } |y|\leq \frac{1}{\rho}\\
\max\left((-y+\frac{1}{\rho})^{1/3},-R_x\right),  & \text{ if } y>\frac{1}{\rho} \text{ and } y\leq R_y \\
\min\left((-y-\frac{1}{\rho})^{1/3}, R_x\right), & \text{ if } y<-\frac{1}{\rho} \text{ and } y\geq  -R_y~. 
\end{cases}
\end{equation*}
We have $|\arg\min_{x}F(x,y_1)-\arg\min_{x}F(x,y_2)|\leq |y_1-y_2|^{1/3}$. The parallel argument of the H{\"{o}}lder continuity can also be made in terms of $\arg\min_{y}F(x,y)$.
\end{proof}

\paragraph{Example 2} The problem $F(x,y)=\frac{\rho}{2n}x^{2n}+|x|+\rho xy-|y|-\frac{\rho}{2n}y^{2n}$, where $n=1,2\ldots$, satisfies Assumptions~\ref{ass:2}, \ref{ass:3} and~\ref{ass:4}.
\begin{proof}
	We can use a similar argument as in Example 1 to show that the problem is strictly-convex-strictly-concave, Assumption~\ref{ass:3} holds with $\theta_1=1$ and Assumption~\ref{ass:4} holds with $\theta_2=\frac{1}{2n-1}$.
\end{proof}

\end{document}